\documentclass[11pt,reqno]{amsart}
\usepackage{amsthm}
\usepackage{amscd}
\usepackage{amsfonts}
\usepackage{amssymb}
\usepackage{amsgen}
\usepackage{amsmath}
\usepackage{amsopn}
\usepackage{verbatim}
\usepackage{xypic}
\usepackage{xspace}
\usepackage{multicol}
\usepackage{url}
\usepackage{upref}

\theoremstyle{plain}
\newtheorem{thm}{Theorem}[section]
\newtheorem{lem}[thm]{Lemma}

\newtheorem{cor}[thm]{Corollary}

\theoremstyle{definition}

\newtheorem{defn}[thm]{Definition}

\theoremstyle{remark}
\newtheorem{rem}[thm]{Remark}

 \font\cyr=wncyr10

\DeclareMathOperator{\sgn}{{sgn}}

 \newcommand{\nc}{\newcommand}
 \nc{\per}[1]{\underset{#1}{\boldsymbol \pi}\,}

 \nc{\MT}{{\rm MT}}
 \nc{\XX}{{X}}
 \nc{\gF}{{\varPhi}}
 \nc{\ot}{\otimes}
 \nc{\bt}{{\bf 2}}

 \nc{\wht}{\widehat}
 \nc{\bwg}{{\bigwedge}}
 \nc{\wg}{{\wedge}}
 \nc{\mal}{{{\scriptstyle \maltese}}}
 \nc{\fA}{{\mathfrak A}}
 \nc{\HH}{{\mathfrak H}}
 \nc{\ra}{\rightarrow}
 \nc{\ors}{{\bfs}}
 \nc{\orr}{{\bfr}}
 \nc{\os}{{\overset}}
 \nc{\G}{{\mathbb G}}
 \nc{\F}{{\mathbb F}}
 \nc{\Z}{{\mathbb Z}}
 \nc{\R}{{\mathbb R}}
 \nc{\N}{{\mathbb N}}
 \nc{\ZN}{{\mathbb Z_{\ge 0}}}
 \nc{\Q}{{\mathbb Q}}
 \nc{\C}{{\mathbb C}}
 \nc{\CP}{{\mathbb{CP}}}
 \nc{\Cnn}{{\mathbb C}_{\ge 0}}
 \nc{\Cp}{{\mathbb C}_{>0}}
 \nc{\MPV}{{\mathcal{MPV}}}
 
 \nc{\tB}{{\tilde B}}

 \nc{\ol}{\overline}

 \nc{\oI}{{\ol{I}}}
 \nc{\bI}{{\bar{I}}}

 \nc{\suf}{{\ast\,}}
 \nc{\sufq}{{\ast_q\,}}
 \nc{\gam}{{\gamma}}
 \nc{\gG}{{\Gamma}}
 \nc{\om}{{\omega}}
 \nc{\vep}{{\varepsilon}}
 \nc{\ga}{{\alpha}}
 \nc{\gl}{{\lambda}}
 \nc{\gb}{{\beta}}
 \nc{\gf}{{\varphi}}
 \nc{\gd}{{\delta}}
 \nc{\orgd}{{\vec \gd\,}}
 \nc{\gs}{{\sigma}}
 \nc{\gth}{{\theta}}
 \nc{\gS}{{\Sigma}}

 \nc{\gk}{{\kappa}}
  \nc{\gz}{{\zeta}}
 \nc{\tgz}{{\tilde{\zeta}}}
 \nc{\gO}{{\Omega}}
 \nc{\sif}{{\mathcal S}}
 \nc{\gt}{{\tau}}
 \nc{\Lra}{\Longrightarrow}
 \nc{\lra}{\longrightarrow}
 \nc{\lmaps}{\longmapsto}
 \nc{\fS}{{\mathfrak S}}
 \nc{\DD}{{\mathfrak D}}
 \nc{\Llra}{\Longleftrightarrow}
 \nc{\ola}{\overleftarrow}
 \nc{\lms}{\longmapsto}
 \nc{\cv}{{{\mathsf c}{\mathsf v}}}
 \nc{\zq}{{\zeta_q}}
 \nc\qup{{q\uparrow 1}}
 \nc{\us}{\underset}
 \nc{\tn}{{\tilde{n}}}
 \nc{\gD}{{\Delta}}
 \nc{\bi}{{\bf i}}
 \nc{\bfone}{{\bf 1}}

 \nc{\bfa}{{\bf a}}
 \nc{\bfb}{{\bf b}}
 \nc{\bfc}{{\bf c}}
 \nc{\bfd}{{\bf d}}
 \nc{\bfe}{{\bf e}}
 \nc{\bff}{{\bf f}}
 \nc{\bfg}{{\bf g}}
 \nc{\bfi}{{\bf i}}
 \nc{\bfj}{{\bf j}}

 \nc{\bfA}{{\bf A}}
 \nc{\bfn}{{\bf n}}
 \nc{\bfl}{{\bf l}}
 \nc{\bfk}{{\bf k}}
 \nc{\bfm}{{\bf m}}
 \nc{\bfo}{{\bf o}}
 \nc{\bfp}{{\bf p}}
 \nc{\bfq}{{\bf q}}
 \nc{\bfr}{{\bf r}}
 \nc{\bfs}{{\bf s}}
 \nc{\bft}{{\bf t}}
 \nc{\bfu}{{\bf u}}
  \nc{\tbfs}{{\tilde{\bfs}}}
 \nc{\tbft}{{\tilde{\bft}}}
 \nc{\tbfu}{{\tilde{\bfu}}}
 \nc{\ttbfs}{{\hat{\bfs}}}
 \nc{\ttbft}{{\hat{\bft}}}
 \nc{\ttbfu}{{\hat{\bfu}}}

 \nc{\mmu}{{\hat{\mu}}}
 \nc{\ggk}{{\hat{\gk}}}
 \nc{\ggs}{{\hat{\gs}}}
 \nc{\bfv}{{\bf v}}
  \nc{\ttbfv}{{\hat{\bfv}}}
 \nc{\bfw}{{\bf w}}
 \nc{\bfx}{{\bf x}}
 \nc{\bfy}{{\bf y}}
 \nc{\bfz}{{\bf z}}
 \nc{\bfB}{{\bf B}}
 \nc{\bfP}{{\bf P}}
 \nc{\bfQ}{{\bf Q}}
 \nc{\bfY}{{\bf Y}}
 \nc{\bfgb}{{\boldsymbol \gb}}
 \nc{\bfgl}{{\boldsymbol \gl}}
 \nc{\wbfgl}{{\widetilde{\bfgl}}}
 \nc{\wwbfgl}{{\overset{\text{\raisebox{-2pt}{$\approx$}}}{\bfgl}}}
 \nc{\bfga}{{\boldsymbol \ga}}
 \nc{\bfrho}{{\boldsymbol \rho}}
 \nc{\bfchi}{{\boldsymbol \chi}}
 \nc{\QX}{{\Q\langle \bfX\rangle}}
 \nc{\QY}{{\Q\langle \bfY\rangle}}
 \nc{\CX}{{\C\langle \bfX\rangle}}
 \nc{\CY}{{\C\langle \bfY\rangle}}
 \nc{\QXX}{{\Q\langle\!\langle \bfX\rangle\!\rangle}}
 \nc{\QYY}{{\Q\langle\!\langle \bfY\rangle\!\rangle}}
 \nc{\CXX}{{\C\langle\!\langle \bfX\rangle\!\rangle}}
 \nc{\CYY}{{\C\langle\!\langle \bfY\rangle\!\rangle}}

 \nc{\bbA}{{\mathbb A}}
 \nc{\bbB}{{\mathbb B}}
 \nc{\bbC}{{\mathbb C}}
 \nc{\bbD}{{\mathbb D}}
 \nc{\bbE}{{\mathbb E}}
 \nc{\bbF}{{\mathbb F}}
 \nc{\bbG}{{\mathbb G}}
 \nc{\bbH}{{\mathbb H}}
 \nc{\bbI}{{\mathbb I}}
 \nc{\bbJ}{{\mathbb J}}
 \nc{\bbK}{{\mathbb K}}
 \nc{\bbL}{{\mathbb L}}
 \nc{\bbM}{{\mathbb M}}
 \nc{\bbN}{{\mathbb N}}
 \nc{\bbO}{{\mathbb O}}
 \nc{\bbP}{{\mathbb P}}
 \nc{\bbQ}{{\mathbb Q}}
 \nc{\bbR}{{\mathbb R}}
 \nc{\bbS}{{\mathbb S}}
 \nc{\bbT}{{\mathbb T}}
 \nc{\bbU}{{\mathbb U}}
 \nc{\bbV}{{\mathbb V}}
 \nc{\bbW}{{\mathbb W}}
 \nc{\bbX}{{\mathbb X}}
 \nc{\bbY}{{\mathbb Y}}
 \nc{\bbZ}{{\mathbb Z}}
 \nc{\bba}{{\mathbb a}}
 \nc{\bbb}{{\mathbb b}}
 \nc{\bbc}{{\mathbb c}}
 \nc{\bbd}{{\mathbb d}}
 \nc{\bbe}{{\mathbb e}}
 \nc{\bbf}{{\mathbb f}}
 \nc{\bbg}{{\mathbb g}}
 \nc{\bbh}{{\mathbb h}}
 \nc{\bbi}{{\mathbb i}}
 \nc{\bbk}{{\mathbb k}}
 \nc{\bbl}{{\mathbb l}}
 \nc{\bbm}{{\mathbb m}}
 \nc{\bbn}{{\mathbb n}}
 \nc{\bbo}{{\mathbb o}}
 \nc{\bbp}{{\mathbb p}}
 \nc{\bbq}{{\mathbb q}}
 \nc{\bbr}{{\mathbb r}}
 \nc{\bbs}{{\mathbb s}}
 \nc{\bbt}{{\mathbb t}}
 \nc{\bbu}{{\mathbb u}}
 \nc{\bbv}{{\mathbb v}}
 \nc{\bbw}{{\mathbb w}}
 \nc{\bbx}{{\mathbb x}}
 \nc{\bby}{{\mathbb y}}
 \nc{\bbz}{{\mathbb z}}

 \nc{\MZV}{{\mathcal{MZV}}}
 \nc{\calA}{{\mathcal A}}
 \nc{\calB}{{\mathcal B}}
 \nc{\calC}{{\mathcal C}}
 \nc{\calD}{{\mathcal D}}
 \nc{\calE}{{\mathcal E}}
 \nc{\calF}{{\mathcal F}}
 \nc{\calG}{{\mathcal G}}
 \nc{\calH}{{\mathcal H}}
 \nc{\calI}{{\mathcal I}}
 \nc{\calJ}{{\mathcal J}}
 \nc{\calK}{{\mathcal K}}
 \nc{\calL}{{\mathcal L}}
 \nc{\calM}{{\mathcal M}}
 \nc{\calN}{{\mathcal N}}
 \nc{\calO}{{\mathcal O}}
 \nc{\calP}{{\mathcal P}}
 \nc{\calQ}{{\mathcal Q}}
 \nc{\calR}{{\mathcal R}}
 \nc{\calS}{{\mathcal S}}
 \nc{\calT}{{\mathcal T}}
 \nc{\calU}{{\mathcal U}}
 \nc{\calV}{{\mathcal V}}
 \nc{\calW}{{\mathcal W}}
 \nc{\calX}{{\mathcal X}}
 \nc{\calY}{{\mathcal Y}}
 \nc{\calZ}{{\mathcal Z}}
  \nc{\cala}{{\mathcal a}}
 \nc{\calb}{{\mathcal b}}
 \nc{\calc}{{\mathcal c}}
 \nc{\cald}{{\mathcal d}}
 \nc{\cale}{{\mathcal e}}
 \nc{\calf}{{\mathcal f}}
 \nc{\calg}{{\mathcal g}}
 \nc{\calh}{{\mathcal h}}
 \nc{\cali}{{\mathcal i}}
 \nc{\calj}{{\mathcal j}}
 \nc{\calk}{{\mathcal k}}
 \nc{\call}{{\mathcal l}}
 \nc{\calm}{{\mathcal m}}
 \nc{\caln}{{\mathcal n}}
 \nc{\calo}{{\mathcal o}}
 \nc{\calp}{{\mathsf p}}
 \nc{\calq}{{\mathcal q}}
 \nc{\calr}{{\mathcal r}}
 \nc{\cals}{{\mathcal s}}
 \nc{\calt}{{\mathcal t}}
 \nc{\calu}{{\mathcal u}}
 \nc{\calv}{{\mathcal v}}
 \nc{\calw}{{\mathcal w}}
 \nc{\calx}{{\mathcal x}}
 \nc{\caly}{{\mathcal y}}
 \nc{\calz}{{\mathcal z}}

 \nc{\frakA}{{\mathfrak A}}
 \nc{\frakB}{{\mathfrak B}}
 \nc{\frakC}{{\mathfrak C}}
 \nc{\frakD}{{\mathfrak D}}
 \nc{\frakE}{{\mathfrak E}}
 \nc{\frakF}{{\mathfrak F}}
 \nc{\frakG}{{\mathfrak G}}
 \nc{\frakH}{{\mathfrak H}}
 \nc{\frakI}{{\mathfrak I}}
 \nc{\frakJ}{{\mathfrak J}}
 \nc{\frakK}{{\mathfrak K}}
 \nc{\frakL}{{\mathfrak L}}
 \nc{\frakM}{{\mathfrak M}}
 \nc{\frakN}{{\mathfrak N}}
 \nc{\frakO}{{\mathfrak O}}
 \nc{\frakP}{{\mathfrak P}}
 \nc{\frakQ}{{\mathfrak Q}}
 \nc{\frakR}{{\mathfrak R}}
 \nc{\frakS}{{\mathfrak S}}
 \nc{\frakT}{{\mathfrak T}}
 \nc{\frakU}{{\mathfrak U}}
 \nc{\frakV}{{\mathfrak V}}
 \nc{\frakW}{{\mathfrak W}}
 \nc{\frakX}{{\mathfrak X}}
 \nc{\frakY}{{\mathfrak Y}}
 \nc{\frakZ}{{\mathfrak Z}}
 \nc{\fraka}{{\mathfrak a}}
 \nc{\frakb}{{\mathfrak b}}
 \nc{\frakc}{{\mathfrak c}}
 \nc{\frakd}{{\mathfrak d}}
 \nc{\frake}{{\mathfrak e}}
 \nc{\frakf}{{\mathfrak f}}
 \nc{\frakg}{{\mathfrak g}}
 \nc{\frakh}{{\mathfrak h}}
 \nc{\fraki}{{\mathfrak i}}
 \nc{\frakj}{{\mathfrak j}}
 \nc{\frakk}{{\mathfrak k}}
 \nc{\frakl}{{\mathfrak l}}
 \nc{\frakm}{{\mathfrak m}}
 \nc{\frakn}{{\mathfrak n}}
 \nc{\frako}{{\mathfrak o}}
 \nc{\frakp}{{\mathfrak p}}
 \nc{\frakq}{{\mathfrak q}}
 \nc{\frakr}{{\mathfrak r}}
 \nc{\fraks}{{\mathfrak s}}
 \nc{\frakt}{{\mathfrak t}}
 \nc{\fraku}{{\mathfrak u}}
 \nc{\frakv}{{\mathfrak v}}
 \nc{\frakw}{{\mathfrak w}}
 \nc{\frakx}{{\mathfrak x}}
 \nc{\fraky}{{\mathfrak y}}
 \nc{\frakz}{{\mathfrak z}}
 \nc{\so}{{\mathfrak so}}
 \nc{\sa}{{\mbox{{\scriptsize \cyr x}}}}
 \nc{\slfour}{{\mathfrak sl}_4}
 \nc{\one}{{\bf 1}}
 \nc{\zero}{{\bf 0}}
 \nc{\Qxy}{\Q\langle x,y\rangle}

 \nc{\DZ}{\mathcal{DZ}}
 \nc{\genf}{\genfrac{[}{]}{0pt}{}}
 \def\sqfr#1#2{{\frac{\genf{#1}{#2}}{\genf{#1+#2}{#2}}}}

\def\HH{\overline{H}}

\nc{\leg}[2]{\left({#1\over #2}\right)}

 \nc{\oll}[1]{\underline{#1}}
 \nc{\hyf}{\text{-}}

 \nc{\bJ}{{\bar{J}}}
 \nc{\bK}{{\bar{K}}}
 \nc{\db}{{\mathbb D}}

 \nc{\hone}{{\widehat{1}}}
 \nc{\wbfp}{{\widetilde{\bfp}}}
 \nc{\wdp}{{\widetilde{p}}}
 \nc{\wwbfp}{{\overset{\text{\raisebox{-2pt}{$\approx$}}}{\bfp}}}
 \nc{\wwdp}{{\overset{\text{\raisebox{-2pt}{$\approx$}}}{\!p}}}
 \nc{\wwB}{{\overset{\text{\raisebox{-2pt}{$\approx$}}}{B}}}
 \nc{\wbfw}{{\widetilde{\bfw}}}
 \nc{\wwbfw}{{\overset{\text{\raisebox{-2pt}{$\approx$}}}{\bfw}}}

\nc{\wdl}{{\widetilde{\gl}}}
 \nc{\wwdl}{{\overset{\text{\raisebox{-2pt}{$\approx$}}}{\!\gl}}}

\begin{document}

\title[On $q$-Analogs of MHS and MZSV Identities]
{On $q$-Analogs of Some Families of Multiple Harmonic Sums and Multiple Zeta Star Value Identities}

% author one information
\author{Kh.~Hessami Pilehrood}
\address{The Fields Institute for Research in Mathematical Sciences, 222 College St, Toronto, ON M5T 3J1,  Canada}
\email{hessamik@gmail.com}

\author{T.~Hessami Pilehrood}
\address{The Fields Institute for Research in Mathematical Sciences, 222 College St, Toronto, ON M5T 3J1,  Canada}
\email{hessamit@gmail.com}

\author{Jianqiang Zhao}
\address{Taida Institute of Mathematical Sciences, National Taiwan University, Taipei, Taiwan, 106}
\email{zhaoj@ihes.fr}


\date{}


\begin{abstract}
In recent years, there has been  intensive research on the $\Q$-linear relations between
multiple zeta (star) values. In this paper, we prove many families of identities
involving the $q$-analog of these values, from which we can always recover the
corresponding classical identities by taking $q\to 1$.
The main result  of the paper is the duality relations between multiple zeta star values and Euler
sums and their $q$-analogs, which are generalizations of the Two-one formula and some
multiple harmonic sum identities and their $q$-analogs proved by the authors recently.
Such duality relations lead to a proof of the conjecture
by Ihara \emph{et al.}\ that the Hoffman $\star$-elements $\zeta^{\star}(s_1,\dots,s_r)$ with
$s_i\in\{2,3\}$ span  the vector space generated by
multiple zeta values over ${\mathbb Q}$.
\end{abstract}

\maketitle

\section{Introduction}
Multiple harmonic sums (MHS) are nested generalizations of harmonic sums and multiple zeta
values (MZV) are the limits of MHS when the number of terms in the sum goes to infinity.
In recent years, MHS, MZV and their generalizations have been found to be intimately related to
Feynman integrals in perturbative quantum field theory
\cite{Broadhurst1996,BrownSc2012,Vermaseren1999} in physics as well as to
Hopf and Lie algebras, combinatorics (double shuffle relations)
\cite{Hoffman1997,Hoffman2004b,Hoffman2005},
algebraic geometry \cite{Brown2012,Goncharov2001a,GoncharovMa2004},
and even modular forms \cite{GKZ2006} in mathematics.

We now recall their basic setup. In order to unify MHS, MZV and their alternating versions
we first define a sort of double cover of the set $\N_0=\N\cup \{0\}$
where $\N$ is the set of positive integers.
\begin{defn} \label{defn:dbANDoplus}
Let $\db_0:=\N_0 \cup \ol{\N}_0$ and $\db:=\N \cup \ol{\N}$  be the sets of
\emph{signed nonnegative} and \emph{signed positive  numbers}, respectively,
 where
\begin{equation*}
  \ol{\N}_0=\{\bar k: k\in\N_0 \} \quad\text{and}\quad \ol{\N}=\{\bar k: k\in\N \}.
\end{equation*}
In some sense, $\bar k$ is $k$ dressed by a negative sign, but $\bar k$ is not a negative number.
Define for all $k\in\N_0$ the absolute value function $| \cdot |$ on $\db_0$ by $|k|=|\bar k|=k$
and the sign function by $\sgn(k)=1$ and $\sgn(\bar k)=-1$. We make $\db_0$ a semi-group by defining
a commutative and associative binary operation $\oplus$ (called \emph{O-plus}) as follows: for all $a,b\in\db_0$
\begin{equation}\label{equ:oplusDefn}
    a\oplus b=
\left\{
  \begin{array}{ll}
    \ol{|a|+|b|}, & \hbox{if only one of $a$ or $b$ is in $\N_0$;} \\
    a+b, & \hbox{if $a,b\in \N_0$;} \\
    |a|+|b|, & \hbox{if  $a,b\in \ol{\N}_0$.} \\
  \end{array}
\right.
\end{equation}
\end{defn}

For $\bfs=(s_1, \ldots, s_m)\in \db^m$, we define the (alternating) multiple harmonic sums by
\begin{equation*}
H_n(\bfs):=\sum_{n\ge k_1>\cdots>k_m\ge 1}
\prod_{j=1}^m\frac{\sgn(s_j)^{k_j}}{k_j^{|s_j|}},\quad \text{and}\quad
H^{\star}_n(\bfs):=\sum_{n\ge k_1\ge\cdots \ge k_m\ge 1}
\prod_{j=1}^m\frac{\sgn(s_j)^{k_j}}{k_j^{|s_j|}}.
\end{equation*}
Correspondingly we can define the (alternating) Euler sums by
\begin{equation} \label{equ:ClassicalStarz}
\zeta(\bfs):=\sum_{k_1>\cdots>k_m\ge 1}
\prod_{j=1}^m\frac{\sgn(s_j)^{k_j}}{k_j^{|s_j|}},\quad \text{and}\quad
\zeta^{\star}(\bfs):=\sum_{k_1\ge\cdots\ge k_m\ge 1}
\prod_{j=1}^m\frac{\sgn(s_j)^{k_j}}{k_j^{|s_j|}}
\end{equation}
where $s_1\ne 1$ in order for the series to converge. If $\bfs\in\N^m$ then
$\zeta(\bfs)$ is called a \emph{multiple zeta value} (MZV) and $\zeta^{\star}(\bfs)$
a  \emph{multiple zeta star value} (MZSV).
We call $\ell(\bfs)=m$ the \emph{length} (or \emph{depth}) and $|\bfs|=|s_1|+\dots+|s_m|$
the \emph{weight} of the string~$\bfs$.
One of the central themes in the study of Euler sums, MZV and MZSV is to
find as many $\Q$-linear relations between these values as possible.
Conjecturally, nontrivial relations can exist only among MZV and MZSV of the same weight.
Following Glanois \cite{Glanois2015}, we define the Euler $\sharp$ sums by
\begin{equation*}
 \zeta^\sharp(\bfs):= \sum_{\bfp=s_1\circ s_2\circ\dots\circ s_d} 2^{\ell(\bfp)} \zeta(\bfp),
\end{equation*}
where $\bfp$ runs through all indices of the form $(s_1\circ s_2\circ\dots\circ s_d)$ with
``$\circ$'' being either the symbol ``,'' or the O-plus ``$\oplus$''.
In \cite{LinebargerZh2013,Zhao2013a} Linebarger and the third author obtained many
families of identities involving
both MHS and MZSV after getting inspiration from \cite{HessamiPilehrood2Ta2013}.
In particular, the third author proved the following so-called Two-one formula conjectured by
Ohno and Zudilin in \cite{OhnoZu2008}:
\begin{thm}\label{thm:MZSV21} \emph{(\cite[Theorem~1.3]{Zhao2013a})}
Let $r\in \N$ and $\bfs=(\{2\}^{a_1},1,\dots,\{2\}^{a_r},1)$
where $a_1\in\N$ and $a_j\in\N_0$ for all $j\ge 2$. Then we have
 \begin{equation*}
\zeta^\star(\bfs) = \zeta^\sharp( 2a_1+1,\dots, 2 a_r+1).
\end{equation*}
\end{thm}

One particularly well-behaved $q$-analog of the multiple zeta 
functions is defined in \cite{Zhao2007c}
by the third author, generalizing the Riemann $q$-zeta function studied by
Kaneko et al.\ \cite{KanekoKuWa2003}. There, again, it is very important to understand the
relations between their special values, see \cite{Bradley2005} for some relevant results.
Recently, the first two authors proved a $q$-analog of the Two-one formula
in \cite{HessamiPilehrood22013}. Our original goal of this paper
was to provide further analogs of the identities contained in \cite{LinebargerZh2013,Zhao2013a}.
However, we have achieved much more because we can now actually treat arbitrary $q$-MZSV and express it
in terms of $q$-analog Euler sums (of the non-star version).
By taking $q\to 1$ we obtain the following result.
\begin{thm} \label{thm:general-MZSV}
Let  $\bfs=(s_1,\dots,s_d)\in \N^d$ with $s_1>1.$  Set $\iota_\bfs =1$ if $s_d=1$ and
$\iota_\bfs =-1$ if $s_d>1$. Suppose there exists
$\bfgl=(\gl_1,\dots,\gl_m)\in\db^m$ (determined uniquely by $\bfs$) such that
there is an expansion of the form
\begin{equation*}
\zeta^{\star}(\bfs)=
   \iota_\bfs   \zeta^\sharp(\gl_1,\gl_2,\dots,\gl_m).
\end{equation*}
Then we have
\begin{enumerate}
\item[\upshape (i)] For any positive integer $a$,
\begin{equation*}
\zeta^{\star}(\{2\}^a, \bfs)=
   \iota_\bfs   \zeta^\sharp(2a\oplus\gl_1,\gl_2,\dots,\gl_m).
\end{equation*}

\item[\upshape (ii)] For any positive integers $a$ and $l$,
\begin{equation*}
\zeta^{\star}(\{2\}^a,\{1\}^l, \bfs)=
   \iota_\bfs   \zeta^\sharp(2a+1,\{1\}^{l-1},\gl_1,\gl_2,\dots,\gl_m).
\end{equation*}

\item[\upshape (iii)] For any positive integers $c\ge 3$ and $l$,
\begin{equation*}
\zeta^{\star}(c,\{1\}^l, \bfs)=
   \iota_\bfs   \zeta^\sharp(\ol{2}, \{1\}^{c-3},\ol{2},\{1\}^{l-1},\gl_1,\gl_2,\dots,\gl_m).
\end{equation*}

\item[\upshape (iv)] For any integers $b\ge 0$ and $c\ge 3$,
\begin{equation*}
\zeta^{\star}(\{2\}^b, c, \bfs)=
 \iota_\bfs   \zeta^\sharp(\ol{2b+2}, \{1\}^{c-3},\gl_1\oplus \ol{1},\gl_2,\dots,\gl_m).
\end{equation*}
\end{enumerate}
\end{thm}

This provides very elegant simplifications when $\bfs$ contains many 2's in it.
The following identity is an illuminating example:
for any $a\in\N$ and  $b,c,d\in\N_0$, we have
\begin{equation}\label{equ:keyExample}
\begin{split}
&\zeta^{\star}(\{2\}^{a},1,\{2\}^{b},1,\{2\}^{c},3,\{2\}^{d},1)\\
=&2\zeta(2a+2b+2c+2d+6) +4\zeta(2a+1,2b+2c+2d+5) \\
      +&4\zeta(2a+2b+2,2c+2d+4)
      +4\zeta(\ol{2a+2b+2c+4},\ol{2d+2}) \\
+&8\zeta(2a+1,2b+1,2c+2d+4)
+8\zeta(2a+1,\ol{2b+2c+3},\ol{2d+2})\\
+&8\zeta(2a+2b+2,\ol{2c+2},\ol{2d+2})
+16\zeta(2a+1,2b+1,\ol{2c+2},\ol{2d+2}).
\end{split}
\end{equation}
We also verified this identity numerically for $1\le a\le 2$ and $0\le b,c,d\le 2$ with EZ-face \cite{EZface}
with errors bounded by $10^{-50}$.

According to Theorem~\ref{thm:general-MZSV}, we can treat arbitrary
MZSV by building up from the three base cases:
$\zeta^{\star}(\{2\}^a)$, $\zeta^{\star}(\{2\}^a, 1)$ ($a\ge 1$)
and $\zeta^{\star}(\{2\}^b, c)$ ($b\ge 0$ and $c\ge 3$)
treated in \cite{HessamiPilehrood2Ta2013}. Here is the general statement.

\begin{thm} \label{thm:general-Glanois}
Let  $a_0,a_j\in \N_0$, $c_j\in \N$ and $c_j\ne 2$ for all $j=1,\dots,d$. Assume $a_0>0$
or $c_1\ge 3$. Set $\gd(c)=1$ if $c=1$ and $\gd(c)=0$ if $c\ge 3$. Moreover, put $\{1\}^{n}=\{1\}^{\max(n,0)}$.
Then we have
\begin{equation} \label{equ:Duality}
\zeta^{\star}(\{2\}^{a_0},c_1,\{2\}^{a_1},\dots,c_d,\{2\}^{a_d})=
 \pm \zeta^\sharp(B_0,\{1\}^{c_1-3},B_1,\dots,\{1\}^{c_d-3},B_d).
\end{equation}
Here the leading sign $\pm$ is $+$ if and only if $a_d=0$ and $c_d=1$, and
\begin{align*}
B_j=
\left\{
  \begin{array}{ll}
    A_j, & \hbox{if $A_j$ is odd;} \\
    \ol{A_j}, & \hbox{if $A_j>0$ and even;} \\
    \text{vacuous}, & \hbox{if $A_j=0$,}
  \end{array}
\right.
\end{align*}
where
\begin{align*}
A_j=
\left\{
  \begin{array}{ll}
    2a_0+2-\gd(c_1) , & \hbox{if $j=0$;} \\
    2a_d+1-\gd(c_d), & \hbox{if $j=d$;} \\
    2a_j+3-\gd(c_j)-\gd(c_{j+1}), & \hbox{if $0<j<d$.}
  \end{array}
\right.
\end{align*}
\end{thm}
Formula (\ref{equ:Duality}) can be considered as a general duality relation which expresses arbitrary multiple
zeta star value in terms of Euler $\sharp$ sums. It generalizes the Two-one formula and many other $2$-$c$-$2$-$c$, $2$-$1$-$2$-$c$, $2$-$c$-$2$-$1$
formulas with $c\ge 3$, proved by the third author  \cite{LinebargerZh2013, Zhao2013a}.  A $q$-analog of Theorem \ref{thm:general-Glanois} is given in Section \ref{S5},
which is Theorem \ref{thm:general-qGlanois}.

In her Ph.D. thesis, Glanois \cite{Glanois2015} studied motivic versions of multiple Euler $\sharp$ sums and proved that the motivic versions of
$$
\zeta^{\sharp}(\overline{2a_0+2}, 2a_1+3, \ldots, 2a_{d-1}+3, 2a_d+1),  \quad\text{with}\quad a_i\ge 0,
$$
form a graded basis of the space of motivic multiple zeta values. As a consequence, by application of the period map, she obtained the following important result.
\begin{thm}[Glanois]  \label{TG}
Each multiple zeta value is a $\Q$-linear combination of elements of the same weight in
$$
\{\zeta^{\sharp}(\overline{2a_0+2}, 2a_1+3, \ldots, 2a_{d-1}+3, 2a_d+1), \quad a_i\ge 0\}.
$$
\end{thm}
Note that Ihara \emph{et al.}\ \cite{Ihara-Kajikawa-Ohno-Okuda2011} conjectured that the Hoffman ${\star}$-elements $\zeta^{\star}(s_1,\dots,s_d)$ with $s_i\in\{2,3\}$
form a basis of the space of MZVs over $\Q$.  Taking into account the Two-three formula, which is a consequence of identity (\ref{equ:Duality}) with $c_1=\dots=c_d=3$,
$$
\zeta^\star(\{2\}^{a_0},3,\{2\}^{a_1}, \ldots, 3, \{2\}^{a_d})=-\zeta^{\sharp}(\overline{2a_0+2}, 2a_1+3, \ldots, 2a_{d-1}+3, 2a_d+1),
$$
and combining it with Theorem \ref{TG}, we get the following statement confirming the conjecture of Ihara \emph{et al.}
\begin{cor}
Every multiple zeta value of weight $w$ is a $\Q$-linear combination of 
the Hoffman ${\star}$-elements $\zeta^\star(s_1, \ldots, s_d)$ with $s_i\in\{2,3\}$ and $\sum s_i=w$.
\end{cor}

In her Ph.D. thesis, Glanois also conjectures that the motivic version of Theorem~\ref{thm:general-Glanois}
should hold (see \cite[Conjecture 4.5.1]{Glanois2015}), whose proof should follow
from Theorem~\ref{thm:general-Glanois} and a Galois descent argument used first
by Brown in \cite{Brown2012} to prove that all the periods of mixed motives
unramified over $\Z$ are $\Q[\frac1{2\pi i}]$-linear combinations of MZVs.
The motivic version of Theorem~\ref{thm:general-Glanois} would imply the motivic version of the Conjecture of
Ihara {\it et al}.\  for the space of motivic multiple zeta values.

\noindent
{\bf Acknowledgement.} KHP and THP gratefully acknowledge support from the  Research Immersion Fellowships of the Fields Institute.
JZ is partly supported by NSF grant DMS-1162116 and the Severo Ochoa Excellence Program at ICMAT.
He also would like to
thank the Max-Planck Institute for Mathematics, the Morningside Center of Mathematics in Beijing
and the Kavli Institute for Theoretical Physics China
for their hospitality where part of this work was done.

\section{Preliminaries and notations}
In this section, we first fix some notation. Throughout the paper let $m$
and $n$ denote nonnegative
integers and $q$ a real number with $0<q<1$. For any real number $a$, put
\begin{equation*}
(a)_0:=(a;q)_0:=1, \qquad (a)_n:=(a;q)_n:=\prod_{k=0}^{n-1}(1-aq^k), \quad n\ge 1.
\end{equation*}
As a convention, throughout the paper we always use $[\ ]$ to denote $q$-analog objects.
For example, the $q$-analog of a positive integer $n$ is given by
\begin{equation*}
[n]=[n]_q:=\sum_{k=0}^{n-1}q^k=\frac{1-q^n}{1-q},
\end{equation*}
and the Gaussian $q$-binomial coefficient
\begin{equation*}
\genf{n}{m}:=\begin{cases}
{\displaystyle \frac{(q)_n}{(q)_{m}(q)_{n-m}},} \quad\qquad &\text{if} \,\,\, 0\le m\le n, \\
\qquad 0, \quad\qquad & \text{otherwise.}
\end{cases}
\end{equation*}
For $m\in\N_0$ and $\bfs=(s_1, \ldots, s_m)\in\db^m_0$, we set $\bfs=\emptyset$ if $m=0$
and define the $q$-analogs of multiple harmonic (star) sums ($q$-MHS)
\begin{equation*}
H_n[\bfs]:=\sum_{n\ge k_1>\cdots>k_m\ge 1}
\prod_{j=1}^m\frac{\sgn(s_j)^{k_j} q^{k_j}}{[k_j]^{|s_j|}} \quad \text{and}\quad
H^{\star}_n[\bfs]:=\sum_{n\ge k_1\ge\cdots \ge k_m\ge 1}
\prod_{j=1}^m\frac{\sgn(s_j)^{k_j} q^{k_j}}{[k_j]^{|s_j|}},
\end{equation*}
with the convention that $H_n[\bfs]=0$ if $n<m$, and
$H^{\star}_n[\emptyset]=H_n[\emptyset]=1$ for all $n\ge 0$. Notice that we allow
$s_j$ to be $0$ or $\bar 0$ in these $q$-MHS.

Now we fix a symbol $\theta$ and define for any $r\in\Z\cup\{\theta\}$ and
$k\in\N$ the function
\begin{equation*}
   Q(r,k):=\left\{
            \begin{array}{ll}
              rk(k-1)/2, & \hbox{if $r>0$;} \\
              rk(k-1)/2-k, & \hbox{if $r\le 0$;} \\
              0, & \hbox{if $r=\theta$.}
            \end{array}
          \right.
\end{equation*}
For $\bfs=(s_1,\dots,s_m)\in\db_0^{m}$, $\bft=(t_1,\ldots, t_m)\in(\N_0)^m$, and
$\bfr=(r_1,\ldots,r_m)\in(\Z\cup\{\theta\})^m$, we
define the mollified companion of $H_n[\bfs]$ by
\begin{equation}\label{equ:calHn}
  \calH_n[\bfs;\bft;\bfr]:=
\sum_{n\ge k_1>\cdots >k_m\ge 1} \sqfr{n}{k_1}\,
 \prod_{j=1}^m  \frac{q^{t_jk_j+Q(r_j,k_j)}(1+q^{k_j})}{\sgn(s_j)^{k_j} [k_j]^{|s_j|}}.
\end{equation}
We call $[\bfs;\bft;\bfr]$ an \emph{admissible triple of mollifiers} if the limit
$\lim\limits_{n\to\infty}\calH_n[\bfs; \bft; \bfr]$ exists.

\begin{defn}\label{defn:squarePlus}
Let $\square$-\emph{plus} $\boxplus$ be a binary operation on
$\Z\cup \{\theta\}$ such that
\begin{itemize}
  \item  $\theta \boxplus a=a\boxplus \theta=a$ for all $a\in \Z\cup \{\theta\}$,
  \item  $a\boxplus b=a+b$ for all $a,b\in\Z$ with $(a,b)\ne (1,-1),(-1,1)$, and
  \item  $1\boxplus (-1)=\theta$ and $(-1)\boxplus 1=0$.
\end{itemize}
\end{defn}
\begin{lem} \label{lem:boxplus}
Let $r\in \{\theta \}\cup\Z\setminus\{0\},$ $d\in\{0, -1\}$. Then for any $k\in \Z$
\begin{align}
Q(r,k)+Q(1,k)&=Q(r\boxplus 1,k), \label{equ:Qr}\\
k^2+Q(d,k)&=Q(2\boxplus d,k). \label{equ:2rk}
\end{align}
Moreover, the projection
\begin{align*}
\pi:(\Z\cup\{\theta\},\boxplus)\lra \, & (\Z,+) \\
 a \lms \, & a            \quad \forall a\in \Z,\\
 \theta \lms \, & 0,
\end{align*}
is a homomorphism of semi-groups and its restriction to $\Z^*$ is injective.
\end{lem}
\begin{proof} Clear.

\end{proof}
For an admissible triple of mollifiers $[\bfs;\bft;\bfr]$, we define
\begin{equation*}
\{s_1\circ \dots\circ s_m;
t_1\circ \dots\circ t_m; r_1\circ \dots\circ r_m\}
\end{equation*}
to be the set of triples of strings produced by replacing every
$\circ$ in $\bfs$ by either comma ``,'' or O-plus ``$\oplus$'', replacing every
$\circ$ in $\bft$  by either comma ``,'' or the usual plus ``$+$'', and
replacing every $\circ$  in $\bfr$ by either comma ``,'' or $\square$-plus ``$\boxplus$''. Moreover,
the commas should be at the same positions for all  $\bfs$, $\bft$ and~$\bfr$.
Now we set
\begin{equation*}
\calH_n^\sharp[\bfs;\bft;\bfr]:=\sum_{(\bfp;\wbfp;\wwbfp)\in
\{s_1\circ \dots\circ s_m; t_1\circ \dots\circ t_m; r_1\circ \dots\circ r_m\}}
\calH_n[\bfp;\wbfp;\wwbfp].
\end{equation*}

In the above  notation, the Two-one formulas for $q$-MHS obtained in
\cite{HessamiPilehrood22013} have the form
\begin{equation} \label{equ:H2a}
H_n^{\star}[\{2\}^a]=-\calH_n[\ol{2a};a;1]
\end{equation}
and, for $a_{\ell+1}\ne 0$
\begin{equation} \label{equ:H2a1}
\begin{split}
H_n^{\star}[\{2\}^{a_1},1,\dots,\{2\}^{a_\ell},1]
=\phantom{-} \calH_n^\sharp[2& a_1+1,\dots,2a_\ell+1;\\
    & a_1+1,\dots,a_\ell+1; 2, \{0\}^{\ell-1}],\\
H_n^{\star}[\{2\}^{a_1},1,\ldots,\{2\}^{a_\ell},1,\{2\}^{a_{\ell+1}}]
=-\calH_n^\sharp[ 2&a_1+1,\dots,2a_\ell+1,\ol{2a_{\ell+1}};\\
   & a_1+1,\dots,a_\ell+1,a_{\ell+1}; 2, \{0\}^{\ell-1},-1].
\end{split}
\end{equation}

Finally, we define the $q$-analog of \emph{multiple zeta values}, $q$-MZV for short, and
$q$-analog of \emph{multiple zeta star values}, or $q$-MZSV, as
\begin{equation*}
\zeta[\bfs]:=\sum_{k_1>\cdots>k_m\ge 1} \prod_{j=1}^m
    \frac{\sgn(s_j)^{k_j}q^{k_j}}{[k_j]^{|s_j|}},\quad\text{and}\quad
\zeta^{\star}[\bfs]:=\sum_{k_1\ge\cdots\ge k_m\ge 1} \prod_{j=1}^m
    \frac{\sgn(s_j)^{k_j} q^{k_j}}{[k_j]^{|s_j|}},
\end{equation*}
respectively. The \emph{mollified companion} of $\zeta[\bfs]$ associated
with the admissible triple of mollifiers $[\bfs;\bft;\bfr]$ is defined by
\begin{equation*}
\begin{split}
\frakz[\bfs;\bft;\bfr]:=&\frakz[s_1,\ldots,s_m; t_1,\ldots,t_m; r_1,\ldots, r_m] \\
= &\sum_{k_1>\cdots >k_m\ge 1}
 \prod_{j=1}^m  \frac{\sgn(s_j)^{k_j} q^{t_jk_j+Q(r_j,k_j)}(1+q^{k_j})}{[k_j]^{|s_j|}},
\end{split}
\end{equation*}
and its $\sharp$-version is defined by
\begin{equation*}
\frakz^\sharp[\bfs;\bft;\bfr]:=\sum_{(\bfp;\wbfp;\wwbfp)\in
\{s_1\circ \dots\circ s_m; t_1\circ \dots\circ t_m; r_1\circ \dots\circ r_m\}}
\frakz[\bfp;\wbfp;\wwbfp].
\end{equation*}
If $m=0$, we put $\zeta^{\star}[\emptyset]=\frakz[\emptyset;\emptyset;\emptyset]=1.$
Throughout the paper the triples of mollifiers $[\bfs; \bft; \bfr]$ are chosen in such a way that
the above multiple series always converges. Notice that in \cite{HessamiPilehrood22013} the
mollified companions of $\zeta[\bfs]$ are defined similarly.

Although our primary goal is to prove $q$-MZSV identities, throughout the paper
we will always work with binomial identities for $q$-MHS first.
To obtain the corresponding $q$-MZSV identities, we need the next result.

\begin{lem} \label{lem:last} \emph{(\cite[Lemma 4.1]{HessamiPilehrood22013})}
Let $0<q<1$, $c, c_1, c_2\in {\mathbb R}$, $c>0$, and let $R_k$ be a sequence of real
numbers satisfying $|R_k|<k^{c_1}q^{c_2k}$ for all $k=1,2,\ldots$. Then
\begin{equation*}
\lim_{n\to\infty}\sum_{k=1}^nq^{ck^2}\left(1-\sqfr{n}{k}\right)R_k=0.
\end{equation*}
\end{lem}

\section{$q$-binomial identities}
The following two combinatorial identities have been proved by
the first two authors using $q$-WZ method.
\begin{lem} \label{lem:l1} \emph{(\cite[Lemma 2.1]{HessamiPilehrood22013})}
For integers $n\ge 1$ and $l\ge 0$, we have
\begin{align}
\sum_{k=l+1}^n(1+q^k)\frac{\genf{n}{k}}{\genf{n+k}{k}}(-1)^{k}
q^{k(k-1)/2}&=\frac{[l]-[n]}{[n]}\frac{\genf{n}{l}}{\genf{n+l}{l}}(-1)^{l}
q^{l(l-1)/2}, \label{equ:i1}\\
\sum_{k=l+1}^n(1+q^k)\,\frac{[k]\genf{n}{k}}{\genf{n+k}{k}}
q^{k(k-1)}&=([n]-[l])\frac{\genf{n}{l}}{\genf{n+l}{l}}\, q^{l^2} \label{equ:i2}.
\end{align}
\end{lem}

The next lemma is the $q$-analog of \cite[Lemma 2.1]{Zhao2013a}.
\begin{lem}\label{lem:combinatorial}
Let $a\in \db_0$, $b\in \N_0$, $c\in\N$, $r\in \{\theta \}\cup\Z\setminus\{0\}$,
and $[\bfx;\bfy;\bfz]$ an admissible triple of mollifiers.
Then for any positive integer $n$,
\begin{equation*}
 \frac{1}{[n]^c}\calH_n[a,\bfx; b, \bfy; r\boxplus 1, \bfz]
=\sum_{(\bfp;\wbfp;\wwbfp)\in\{\ol{0}\circ 1^{\circ(c-1)}\circ (a \oplus \ol{1});
    \, 0^{\circ c}\circ\, b;\, 1\circ\, \theta^{\circ (c-1)}\circ r\}}
    \calH_n[\bfp, \bfx; \wbfp, \bfy; \wwbfp, \bfz].
\end{equation*}
\end{lem}
\begin{proof}
We prove the lemma by induction on $c$. Suppose strings $\bfx, \bfy, \bfz$ have length $m$. Set
\begin{equation} \label{equ:ank}
A_{n,k}=(-1)^{k}(1+q^k)  q^{k(k-1)/2}\frac{\genf{n}{k}}{\genf{n+k}{k}}.
\end{equation}
Then by \eqref{equ:i1},
\begin{equation}\label{equ:sumAnk}
\frac{(1+q^l)}{[l]}\sum_{k=l+1}^n A_{n,k}=\left(\frac{1}{[n]}-\frac{1}{[l]}\right)A_{n,l}
\end{equation}
which, together with~\eqref{equ:Qr} yields
\begin{equation*}
\begin{split}
&\quad\calH_n[\ol{0},a \oplus \ol{1},\bfx; 0,b, \bfy; 1, r, \bfz] \\[5pt]
&=
\sum_{n\ge k>k_0>k_1>\cdots>k_m\ge 1}\!\!
    \frac{A_{n,k}q^{bk_0+Q(r,k_0)}(1+q^{k_0})}{(-\sgn(a))^{k_0} [k_0]^{|a|+1}}\prod_{j=1}^m
    \frac{q^{y_jk_j+Q(z_j,k_j)}(1+q^{k_j})}{ \sgn(x_j)^{k_j}[k_j]^{|x_j|}}\\[3pt]
&=\sum_{n\ge k_0>k_1>\cdots>k_m\ge 1}\frac{q^{bk_0+Q(r,k_0)}(1+q^{k_0})}{(-\sgn(a))^{k_0} [k_0]^{|a|+1}}\prod_{j=1}^m
\frac{q^{y_jk_j+Q(z_j,k_j)}(1+q^{k_j})}{ \sgn(x_j)^{k_j}[k_j]^{|x_j|}}
\sum_{k=k_0+1}^n A_{n,k} \\[3pt]
&=
\frac{1}{[n]}\calH_n[a,\bfx; b, \bfy; r\boxplus 1, \bfz]-\calH_n[a \oplus 1,\bfx; b, \bfy; r\boxplus 1, \bfz].
\end{split}
\end{equation*}
This proves the lemma for $c=1.$ Now suppose $c>1.$ By the case $c=1$, we have
just proved,
\begin{equation*}
\begin{split}
\frac{1}{[n]^c}&\calH_n[a,\bfx; b, \bfy; r\boxplus 1, \bfz]=
\frac{1}{[n]^{c-1}}\left(\frac{1}{[n]}\calH_n[a,\bfx; b, \bfy; r\boxplus 1, \bfz]\right) \\
&=\frac{1}{[n]^{c-1}}\calH_n[a\oplus 1,\bfx; b, \bfy; r\boxplus 1, \bfz]
+\frac{1}{[n]^{c-1}}\calH_n[\ol{0},a\oplus \ol{1},\bfx; 0,b, \bfy; 1, r, \bfz].
\end{split}
\end{equation*}
For the first summand,  we now apply induction assumption using case $c-1$
with $a$ replaced by $a\oplus 1$.
For the second summand, we apply Lemma~\ref{lem:combinatorial} using
case $c-1$ with $a=\ol{0},$ $b=0,$ and $r=\theta$. Then we see
the above is equal to
\begin{equation*}
\begin{split}
&\sum_{(\bfp, \wbfp, \wwbfp)\in \{\ol{0}\circ 1^{\circ(c-2)}\circ\big((a\oplus 1)\oplus\ol{1}\big)
    ;\,  0^{\circ(c-1)}\circ b;
    \,  1\circ \theta^{\circ(c-2)}\circ r\}}
\calH_n[\bfp, \bfx; \wbfp,  \bfy; \wwbfp, \bfz] \\[5pt]
&+
\sum_{(\bfp, \wbfp, \wwbfp)\in\{\ol{0}\circ 1^{\circ(c-1)};
    \, 0^{\circ c};
    \, 1\circ \theta^{\circ(c-1)}\}}
\calH_n[\bfp,  a\oplus\ol{1}, \bfx; \wbfp, b, \bfy; \wwbfp, r, \bfz] \\[5pt]
&=
\sum_{(\bfp, \wbfp, \wwbfp)\in \{\ol{0}\circ 1^{\circ(c-1)}\circ(a\oplus\ol{1});
    \, 0^{\circ c}\circ b;
    \, 1\circ \theta^{\circ (c-1)}\circ r\}}
\calH_n[\bfp, \bfx; \wbfp, \bfy; \wwbfp, \bfz],
\end{split}
\end{equation*}
since $(a\oplus 1)\oplus\ol{1}=1\oplus ( a\oplus\ol{1})$ for all $a\in \db_0$.
We have now completed the proof of the lemma.
\end{proof}

The next corollary is the degenerate case of the proceeding lemma.
\begin{cor}\label{cor:combinatorial}
For all $c\in\N_0$, we have
\begin{equation*}
 \frac{1}{[n]^c}=- \calH_n^\sharp [\ol{0},\{1\}^c; \{0\}^{c+1}; 1, \{\theta\}^c].
\end{equation*}
\end{cor}
\begin{proof}
The case $c=0$ follows from  \eqref{equ:i1} by setting $l=0$.
For $c\ge 1$, using the $c=0$ case, we get
\begin{equation*}
  \frac1{[n]^c}=-\frac1{[n]^{c}}\calH_n[\bar 0; 0; 1]=
-\sum_{(\bfp;\wbfp;\wwbfp)\in \Pi(c)} \calH_n[\bfp; \wbfp; \wwbfp]
\end{equation*}
by taking $a=\bar 0$, $b=0$, $r=\theta$ and $\bfx=\bfy=\bfz=\emptyset$ in Lemma~\ref{lem:combinatorial}.
Hence the corollary is proved.
\end{proof}

\section{MHS and MZSV identities: $\bt\hyf c$ formula}\label{sec:case2bc}
In this section, we start with some $q$-MHS identities involving
arguments of $(\{2\}^a,c)$-type ($c\ge 3$). This provides one of the base cases
upon which we may build general formulas of $q$-MZSV and MZSV whose arguments
can be any admissible strings of positive integers.

\begin{thm} \label{thm:MHS2c}
Let $\bfs=(\{2\}^a,c)$ with $a, c\in\N_0$ and $c\ge 3$. Then
\begin{equation}\label{equ:H2c}
 H^\star_n[\bfs]=- \calH_n^\sharp[\ol{2a+2},\{1\}^{c-2};a+1,\{0\}^{c-2};1,\{\theta\}^{c-2}].
\end{equation}
\end{thm}
\begin{proof}
We proceed by induction on $n$. Set $\Pi(\bfs)=\{(\ol{2a+2})\circ 1^{\circ(c-2)};(a+1)\circ 0^{\circ(c-2)};1\circ \theta^{\circ(c-2)}\}.$
When $n=1$, we have $H^\star_1(\{2\}^a,c) = q^{a+1}.$
On the other hand,
\begin{equation*}
\sum_{(\bfp; \wbfp; \wwbfp)\in \Pi(\bfs)}\calH_1[\bfp; \wbfp; \wwbfp]=
\calH_1[\ol{2a+c}; a+1; 1]=-q^{a+1},
\end{equation*}
and therefore the formula is true.
Suppose the statement is true  for $n-1$. Then by definition
\begin{equation*}
H^\star_n[\bfs] = \sum_{i=0}^{a} \frac {q^{n(a-i)}}{[n]^{2(a-i)}} H^\star_{n-1}[\{2\}^i, c]
+\frac {q^{n(a+1)}}{[n]^{2a+ c}}.
\end{equation*}
Applying  inductive hypothesis, we obtain
\begin{equation} \label{hyp}
H^\star_n[\bfs]=
 -\sum_{i=0}^{a} \frac {q^{n(a-i)}}{[n]^{2(a-i)}}
 \sum_{(\bfp;\wbfp; \wwbfp)\in\Pi(\{2\}^i, c)}\calH_{n-1}[\bfp; \wbfp; \wwbfp]+\frac{q^{n(a+1)}}{[n]^{2a+ c}}.
\end{equation}
Set $\Pi(\bfu_{-1})=\{0\circ 1^{\circ(c-2)};\, 0^{\circ(c-1)};\, \theta^{\circ(c-1)} \}.$
To save space, for any string $\bfgl=(\gl_1,\dots,\gl_m)$, we write the substring $\bfgl_\hone=(\gl_2,\dots,\gl_m)$.
Then the inner sum in \eqref{hyp} becomes
\begin{equation*}
\begin{split}
\sum_{(\bfp;\wbfp; \wwbfp)\in\Pi(\{2\}^i, c)}\calH_{n-1}&[\bfp; \wbfp; \wwbfp]
=\sum_{(\bfp;\wbfp; \wwbfp)\in\Pi(\bfu_{-1})}\!\calH_{n-1}[\ol{2i+2+p}_1,\bfp_\hone;
        i+1+\wdp_1, \wbfp_\hone; 1\boxplus \wwdp_1,\wwbfp_\hone] \\
&=\sum_{(\bfp;\wbfp; \wwbfp)\in\Pi(\bfu_{-1})}\,\sum_{n> k_1>\ldots>k_m\ge 1}
    \frac{q^{k_1(i+1+\wdp_1)}A_{n-1,k_1}}{[k_1]^{2i+2+p_1}}\,
    \prod_{j=2}^m\frac{q^{\wdp_j k_j}(1+q^{k_j})}{[k_j]^{p_j}},
\end{split}
\end{equation*}
where $A_{n,k}$ is defined in \eqref{equ:ank}. Plugging this into
\eqref{hyp} and summing over $i$ by the formula
\begin{equation} \label{equ:geomSum}
A_{n-1,k}\sum_{i=0}^{a}\frac{[n]^{2i}}{[k]^{2i}}\,q^{(k-n)i}=
A_{n,k}\left(\frac{[n]^{2a}}{[k]^{2a}}\,q^{(k-n)a}-\frac{[k]^2}{[n]^2}\,q^{n-k}\right),
\end{equation}
we obtain
\begin{equation*}
\begin{split}
H^\star_n[\bfs]&=-\sum_{(\bfp;\wbfp; \wwbfp)\in\Pi(\bfu_{-1})}\,
    \sum_{n\ge k_1>\cdots>k_m\ge 1}\frac{q^{k_1(\wdp_1+a+1)}A_{n,k_1}}{[k_1]^{2a+2+p_1}}\,
    \prod_{j=2}^m\frac{q^{\wdp_j k_j}(1+q^{k_j})}{[k_j]^{p_j}} \\[3pt]
&+\frac{q^{n(a+1)}}{[n]^{2a+2}}
    \sum_{(\bfp;\wbfp; \wwbfp)\in\Pi(\bfu_{-1})}\,\sum_{n\ge k_1>\cdots>k_m\ge 1}
    \frac{q^{k_1\wdp_1}A_{n,k_1}}{[k_1]^{p_1}}\,
    \prod_{j=2}^m\frac{q^{\wdp_j k_j}(1+q^{k_j})}{[k_j]^{p_j}}
 + \frac{q^{n(a+1)}}{[n]^{2a+c}},
\end{split}
\end{equation*}
which implies
\begin{multline*}
H^\star_n[\bfs]=-\sum_{(\bfp;\wbfp; \wwbfp)\in\Pi(\bfs)}\calH_n[\bfp; \wbfp; \wwbfp]
+\frac{q^{n(a+1)}}{[n]^{2a+2}}\sum_{(\bfp;\wbfp; \wwbfp)\in\Pi(\ttbfu_{-1})}\calH_n[\bfp; \wbfp; \wwbfp]
    +\frac{q^{n(a+1)}}{[n]^{2a+c}},
\end{multline*}
where $\Pi(\ttbfu_{-1})=\{\ol{0}\circ 1^{\circ(c-2)};\, 0^{\circ(c-1)};\, 1\circ\theta^{\circ(c-2)}\}.$
Hence the theorem follows from Corollary~\ref{cor:combinatorial} immediately by replacing $c$ by $c-2$ there.
\end{proof}

\section{MHS and MZSV identities: general case}
\label{S5}

In this section, we prove some general rules which explain what to expect when
we add strings $(\{2\}^a,\{1\}^l)$ or $(\{2\}^b,c)$ to a string of positive integer arguments.
This allows us to extend expansion formulas from the  three base cases \eqref{equ:H2a},
\eqref{equ:H2a1} and \eqref{equ:H2c}  to every string that contains an arbitrary
number of repetitions of $(\{2\}^b,c)$, $(c,\{1\}^l)$ ($b\ge 0$, $c\ge 3$, $l\ge 1$) and
$(\{2\}^a,\{1\}^l)$ ($a\ge 0$ except at the leading position when $a\ge 1$).
For example, $(3,1,2,7,1,1,5,2,2,4)$
can be written as $(\{2\}^0,3,\{2\}^0,1,\{2\}^1,7,\{2\}^0,\{1\}^2,\{2\}^0,5,\{2\}^2,4)$.

For any string $\bfgl=(\gl_1,\dots,\gl_m)$, we set
$\bfgl{}^\circ=\gl_1\circ\dots\circ\gl_m$ and $\bfgl_{\widehat{1}}=(\gl_2,\dots,\gl_m)$.

\begin{thm} \label{thm:general}
Let $n\in\N$ and  $\bfs=(s_1,\dots,s_d)$
be a string of positive integers.
Set $\iota_\bfs =1$ if $s_d=1$ and
$\iota_\bfs =-1$ if $s_d>1$.
Suppose $\bfs$ uniquely determines
the triple of mollifiers $[\bfgl;\wbfgl;\wwbfgl]=[\gl_1,\dots,\gl_m;\wdl_1,\dots,\wdl_m;\wwdl_1,\dots,\wwdl_m]$
satisfying
\begin{equation}\label{equ:conditionWwgl}
\wwdl_1\boxplus\dots\boxplus\wwdl_j\in\{1, 2\} \quad \forall  j\ge 1,
\end{equation}
such that there is an expansion of the form
\begin{equation*}
H_n^{\star}[\bfs]=\iota_\bfs  \calH_n^\sharp[\bfgl;\wbfgl;\wwbfgl].
\end{equation*}
Then for any integers $a,b\ge 0$ and $c\ge 3$, we have
\begin{align} \label{equ:Rule2a}
H_n^{\star}[\{2\}^a,  \bfs]=&\,\iota_\bfs  \calH_{n}^\sharp[2a\oplus\lambda_1,\bfgl{}_{\widehat{1}};
a+\wdl_1,\wbfgl{}{}_{\widehat{1}}; \wwbfgl], \\
\label{equ:Rule2a1}
H_n^{\star}[\{2\}^a, 1, \bfs]=&\,\iota_\bfs  \calH_{n}^\sharp[2a+1,\bfgl;
a+1,\wbfgl; 2,\wwdl_1\boxplus -2,\wwbfgl{}_{\widehat{1}}], \\
\label{equ:Rule2bc}
H_n^{\star}[\{2\}^b, c, \bfs]=&\,\iota_\bfs \calH_{n}^\sharp[
\ol{2b+2},\{1\}^{c-3},\gl_1\oplus \bar{1},\bfgl_{\widehat{1}};
b+1,\{0\}^{c-3},\wbfgl{};
1,\{\theta\}^{c-3},\wwdl_1\boxplus -1,\wwbfgl{}_{\widehat{1}}].
\end{align}
Moreover, in all the index sets appearing on the right hand side above,
the third components still satisfy \eqref{equ:conditionWwgl}.
\end{thm}

Notice that condition \eqref{equ:conditionWwgl} essentially guarantees
that all the triples of mollifiers considered in the paper are admissible.

\begin{proof}
Set $\Pi(\bfs)= \{\bfgl^\circ;\, \wbfgl{}^\circ;\, \wwbfgl{}^\circ\}$.
The proof of the identities is by induction on $n+a$ or $n+b$.
When $n=1$ the theorem is clear.
Assume formulas  \eqref{equ:Rule2a} and  \eqref{equ:Rule2a1} are  true for all
$a+n\le N$ where $N\ge 2.$ Suppose now we have $n\ge 2$ and $n+a=N+1.$ Set

$\Pi_{\bt^a}=\{(2a\oplus\lambda_1)\circ\bfgl{}_{\widehat{1}}^\circ;\,
(a+\wdl_1)\circ\wbfgl{}{}_{\widehat{1}}^\circ;\, \wwbfgl{}^\circ\}$,

$\Pi_{\bt^a 1}=\{(2a+1)\circ\bfgl^\circ;\,
(a+1)\circ\wbfgl{}^\circ;\, 2\circ(\wwdl_1\boxplus -2)\circ(\wwbfgl{}_{\widehat{1}})^\circ\}$,

$\Pi_{\bt^bc}=\{(\ol{2b+2})\circ 1^{\circ(c-3)}\circ(\gl_1\oplus \bar{1})\circ\bfgl^\circ_{\widehat{1}}; \,
                 (b+1)\circ 0^{\circ(c-3)}\circ\wbfgl{}^\circ; \,
               1\circ\theta^{\circ(c-3)}\circ(\wwdl_1\boxplus -1)\circ(\wwbfgl{}_{\widehat{1}})^\circ\}.$

\noindent
We start proving the first identity. By definition, we have
\begin{equation*}
H_n^{\star}[\{2\}^a,\bfs]=\sum_{i=1}^{a}\frac{q^{n(a-i)}}{[n]^{2a-2i}}\,H_{n-1}^{\star}[\{2\}^i,\bfs]+\frac{q^{na}}{[n]^{2a}}\,H_n^{\star}[\bfs].
\end{equation*}
Applying induction assumption, we obtain
\begin{equation}\label{equ:indg}
\iota_\bfs H_n^{\star}[\{2\}^a,\bfs]=\sum_{i=1}^{a}\frac{q^{n(a-i)}}{[n]^{2a-2i}}
    \!\sum_{(\bfp; \wbfp; \wwbfp)\in\Pi_{\bt^i}}\!\calH_{n-1}[\bfp; \wbfp; \wwbfp]
    +\frac{q^{na}}{[n]^{2a}}\sum_{(\bfp; \wbfp; \wwbfp)\in\Pi(\bfs)}\!\calH_n[\bfp; \wbfp; \wwbfp].
\end{equation}
Expanding the inner sum
{\allowdisplaybreaks
\begin{align*}
&\sum_{(\bfp; \wbfp; \wwbfp)\in\Pi_{\bt^i}}\calH_{n-1}[\bfp; \wbfp; \wwbfp]=
    \sum_{(\bfp; \wbfp; \wwbfp)\in\Pi(\bfs)}\calH_{n-1}[2i\oplus p_1, \bfp_\hone; i+\wdp_1,\wbfp_\hone;  \wwbfp] \\
&\qquad=\sum_{(\bfp; \wbfp; \wwbfp)\in\Pi(\bfs)}
    \sum_{n>k_1>\dots>k_r\ge 1}\sqfr{n-1}{k_1}\frac{q^{ik_1}}{[k_1]^{2i}}
    \prod_{j=1}^r\frac{q^{\wdp_jk_j+Q(\wwdp_j,k_j)}(1+q^{k_j})}{\sgn(p_j)^{k_j}\,[k_j]^{|p_j|}}
\end{align*}}
and summing over $i$ in \eqref{equ:indg}, we obtain
{\allowdisplaybreaks
\begin{align*}
\iota_\bfs H_n^{\star}[\{2\}^a,\bfs]=\sum_{(\bfp; \wbfp; \wwbfp)\in\Pi(\bfs)}
    &\sum_{n\ge k_1>\dots>k_r\ge 1}\sqfr{n}{k_1}\left(\frac{q^{ak_1}}{[k_1]^{2a}}-\frac{q^{an}}{[n]^{2a}}\right)
    \prod_{j=1}^r\frac{q^{\wdp_jk_j+Q(\wwdp_j,k_j)}(1+q^{k_j})}{\sgn(p_j)^{k_j}\,[k_j]^{|p_j|}}\\[3pt]
&\,+\frac{q^{na}}{[n]^{2a}}\sum_{(\bfp; \wbfp; \wwbfp)\in\Pi(\bfs)}\!\calH_n[\bfp; \wbfp; \wwbfp],
\end{align*}}
which implies \eqref{equ:Rule2a} by definition and straightforward cancelation.
Similarly, for the second identity, we have by definition
\begin{equation*}
H_n^{\star}[\{2\}^a,1,\bfs]=\sum_{i=0}^{a}\frac{q^{n(a-i)}}{[n]^{2a-2i}}\,H_{n-1}^{\star}[\{2\}^i,1,\bfs]+\frac{q^{n(a+1)}}{[n]^{2a+1}}\,H_n^{\star}[\bfs].
\end{equation*}
Applying induction assumption, we obtain
\begin{equation}\label{indg}
\iota_\bfs H_n^{\star}[\{2\}^a,1,\bfs]=\sum_{i=0}^{a}\frac{q^{n(a-i)}}{[n]^{2a-2i}}
    \!\sum_{(\bfp; \wbfp; \wwbfp)\in\Pi_{\bt^i1}}\!\calH_{n-1}[\bfp; \wbfp; \wwbfp]
    +\frac{q^{n(a+1)}}{[n]^{2a+1}}\sum_{(\bfp; \wbfp; \wwbfp)\in\Pi(\bfs)}\!\calH_n[\bfp; \wbfp; \wwbfp].
\end{equation}
Setting $\Pi_0=\{\bfgl^\circ; \wbfgl{}^\circ; (\wwdl_1\boxplus -2)\circ(\wwbfgl{}_{\widehat{1}})^\circ\}$, we have
{\allowdisplaybreaks
\begin{align*}
&\sum_{(\bfp; \wbfp; \wwbfp)\in\Pi_{\bt^i1}}\calH_{n-1}[\bfp; \wbfp; \wwbfp]=
    \sum_{(\bfp; \wbfp; \wwbfp)\in\Pi_0}\calH_{n-1}[2i+1,\bfp; i+1,\wbfp; 2, \wwbfp] \\
&\qquad +\sum_{(\bfp; \wbfp; \wwbfp)\in\Pi_0}\calH_{n-1}[(2i+1)\oplus p_1,\bfp_\hone;
    i+1+\wdp_1,\wbfp_\hone; 2\boxplus\wwdp_1,\wwbfp_\hone]\\
&=\sum_{(\bfp; \wbfp; \wwbfp)\in\Pi_0}\!\!\left(
    \sum_{n> k_0>k_1>\dots>k_m\ge 1}\!\sqfr{n-1}{k_0}\frac{(1+q^{k_0})q^{(i+1)k_0+k_0(k_0-1)}}{[k_0]^{2i+1}}
    \prod_{j=1}^m\frac{q^{\wdp_jk_j+Q(\wwdp_j,k_j)}(1+q^{k_j})}{\sgn(p_j)^{k_j}\,[k_j]^{|p_j|}} \right. \\
&\qquad \left. +\sum_{n> k_1>\dots>k_m\ge 1}\!\sqfr{n-1}{k_1}
    \frac{(1+q^{k_1})q^{(i+1+\wdp_1)k_1+Q(2\boxplus \wwdp_1,k_1)}}{\sgn(p_1)^{k_1}[k_1]^{2i+1+|p_1|}}
    \prod_{j=2}^m\frac{q^{\wdp_jk_j+Q(\wwdp_j,k_j)}(1+q^{k_j})}{\sgn(p_j)^{k_j}\,[k_j]^{|p_j|}} \right).
\end{align*}}
Substituting the above expression  into \eqref{indg} and summing over $i$ by \eqref{equ:geomSum},
we obtain
{\allowdisplaybreaks
\begin{equation*}
\begin{split}
&\iota_\bfs  H_n^{\star}[\{2\}^a,1,\bfs]-\frac{q^{n(a+1)}}{[n]^{2a+1}}\sum_{(\bfp; \wbfp; \wwbfp)\in\Pi(\bfs)}\calH_{n}[\bfp; \wbfp; \wwbfp] \\
&=\sum_{(\bfp; \wbfp; \wwbfp)\in\Pi_0}\!\!\left(\sum_{n\ge k_0>k_1>\dots>k_m\ge 1}\!
    \sqfr{n}{k_0}\frac{(1+q^{k_0})q^{(a+1)k_0+k_0(k_0-1)}}{[k_0]^{2a+1}}
    \prod_{j=1}^m\frac{q^{\wdp_jk_j+Q(\wwdp_j,k_j)}(1+q^{k_j})}{\sgn(p_j)^{k_j}\,[k_j]^{|p_j|}} \right. \\
&\,-\frac{q^{n(a+1)}}{[n]^{2a+2}}\sum_{n\ge k_0>k_1>\dots>k_m\ge 1}\!\sqfr{n}{k_0}(1+q^{k_0})q^{k_0(k_0-1)}[k_0]
    \prod_{j=1}^m\frac{q^{\wdp_jk_j+Q(\wwdp_j,k_j)}(1+q^{k_j})}{\sgn(p_j)^{k_j}\,[k_j]^{|p_j|}} \\
&\,+\sum_{n\ge k_1>\dots>k_m\ge 1}\!\sqfr{n}{k_1}\frac{(1+q^{k_1})q^{(a+1+\wdp_1)k_1+Q(2\boxplus \wwdp_1,k_1)}}{\sgn(p_1)^{k_1} [k_1]^{2a+1+|p_1|}}
\prod_{j=2}^m\frac{q^{\wdp_jk_j+Q(\wwdp_j,k_j)}(1+q^{k_j})}{\sgn(p_j)^{k_j}\,[k_j]^{|p_j|}}\\
&\left.\,-\frac{q^{n(a+1)}}{[n]^{2a+2}}\sum_{n\ge k_1>\dots>k_m\ge 1}\!\sqfr{n}{k_1}\frac{(1+q^{k_1})q^{\wdp_1k_1+Q(2\boxplus \wwdp_1,k_1)}}{\sgn(p_1)^{k_1}[k_1]^{|p_1|-1}}
\prod_{j=2}^m\frac{q^{\wdp_jk_j+Q(\wwdp_j,k_j)}(1+q^{k_j})}{\sgn(p_j)^{k_j}\,[k_j]^{|p_j|}} \right).
\end{split}
\end{equation*}}
Noticing that the first and third sums on the right-hand side of the above add up to
\begin{equation*}
\sum_{(\bfp; \wbfp; \wwbfp)\in\Pi_{\bt^a1}}\calH_{n}[\bfp; \wbfp; \wwbfp],
\end{equation*}
we have
\begin{equation*}
\begin{split}
&\quad \iota_\bfs H_n^{\star}[\{2\}^a,1,\bfs]-\frac{q^{n(a+1)}}{[n]^{2a+1}}\sum_{(\bfp; \wbfp; \wwbfp)\in\Pi(\bfs)}\calH_{n}[\bfp; \wbfp; \wwbfp]=
\sum_{(\bfp; \wbfp; \wwbfp)\in\Pi_{\bt^a1}}\calH_{n}[\bfp; \wbfp; \wwbfp]
 \\[3pt]
&-\frac{q^{n(a+1)}}{[n]^{2a+2}}\!\sum_{(\bfp; \wbfp; \wwbfp)\in\Pi_0}
    \sum_{n\ge k_1>\dots>k_m\ge 1}\left(
    \prod_{j=1}^m\frac{q^{\wdp_jk_j+Q(\wwdp_j,k_j)}(1+q^{k_j})}{\sgn(p_j)^{k_j}\,[k_j]^{|p_j|}}\!
    \sum_{k_0=k_1+1}^n\!\sqfr{n}{k_0}\frac{(1+q^{k_0})q^{k_0(k_0-1)}}{[k_0]^{-1}}\right.\\[3pt]
&\qquad\qquad\left.+\,
    \sqfr{n}{k_1}\frac{(1+q^{k_1})q^{\wdp_1k_1+Q(2\boxplus \wwdp_1,k_1)}}{\sgn(p_1)^{k_1}[k_1]^{|p_1|-1}}
    \prod_{j=2}^m\frac{q^{\wdp_jk_j+Q(\wwdp_j,k_j)}(1+q^{k_j})}{\sgn(p_j)^{k_j}\,[k_j]^{|p_j|}}\right).
\end{split}
\end{equation*}
Summing the multiple sum in the above over $k_0$ by \eqref{equ:i2} and noticing that for
$(\bfp; \wbfp; \wwbfp)\in\Pi_0$, by \eqref{equ:conditionWwgl},
the first component  $\wwdp_1$ can take
only values $-1$ and $0$, we obtain with the help of \eqref{equ:2rk} that
\begin{equation*}
\iota_\bfs H_n^{\star}[\{2\}^a,1,\bfs]=\sum_{(\bfp; \wbfp; \wwbfp)\in\Pi_{\bt^a1}}\calH_{n}[\bfp; \wbfp; \wwbfp].
\end{equation*}
This proves identity \eqref{equ:Rule2a1} by induction.

Finally, to prove \eqref{equ:Rule2bc}, we proceed by induction on $n+b.$ Assume formula \eqref{equ:Rule2bc} is true for all $b+n\le N$.
Now suppose $b+n=N+1$. By definition, we have
\begin{equation*}
H_n^{\star}[\{2\}^b,c,\bfs]=\sum_{i=0}^{b}\frac{q^{n(b-i)}}{[n]^{2b-2i}}\,H_{n-1}^{\star}[\{2\}^i,c,\bfs]+\frac{q^{n(b+1)}}{[n]^{2b+c}}\,H_n^{\star}[\bfs].
\end{equation*}
By the  induction assumption, we see that
\begin{equation}\label{indg2}
\iota_\bfs H_n^{\star}[\{2\}^b,c,\bfs]=\sum_{i=0}^{b}\frac{q^{n(b-i)}}{[n]^{2b-2i}}\sum_{(\bfp; \wbfp; \wwbfp)\in\Pi_{\bt^ic}}\calH_{n-1}[\bfp; \wbfp; \wwbfp]+\frac{q^{n(b+1)}}{[n]^{2b+c}}\!\sum_{(\bfp; \wbfp; \wwbfp)\in\Pi(\bfs)}\!\calH_n[\bfp; \wbfp; \wwbfp].
\end{equation}
Setting $\Pi_1=\{0\circ 1^{\circ(c-3)}\circ (\gl_1\oplus \bar{1})\circ\bfgl^\circ_{\widehat{1}};
0^{\circ(c-2)}\circ\wbfgl{}^\circ;
\theta^{\circ(c-2)}\circ(\wwdl_1\boxplus -1)\circ(\wwbfgl{}_{\widehat{1}})^\circ\}$,
 we have
\begin{multline*}
 \sum_{(\bfp;\wbfp; \wwbfp)\in\Pi_{\bt^ic}}\calH_{n-1}[\bfp; \wbfp; \wwbfp]
=\sum_{(\bfp;\wbfp; \wwbfp)\in\Pi_1}\!\calH_{n-1}[\ol{2i+2}\oplus p_1, \bfp_\hone; i+1+\wdp_1, \wbfp_\hone; 1\boxplus \wwdp_1,\wwbfp_\hone] \\
 =\!\sum_{(\bfp;\wbfp; \wwbfp)\in\Pi_1}\sum_{n> k_1>\cdots>k_r\ge 1}\!\sqfr{n-1}{k_1}
 \frac{q^{k_1(i+1+\wdp_1)+Q(1\boxplus\wwdp_1,k_1)}}{(-\sgn(p_1))^{k_1}\,[k_1]^{2i+2+|p_1|}}
 \prod_{j=2}^r\frac{q^{\wdp_j k_j+Q(\wwdp_j,k_j)}(1+q^{k_j})}{\sgn(p_j)^{k_j}\,[k_j]^{|p_j|}}.
\end{multline*}
Plugging this into \eqref{indg2} and summing over $i$ by \eqref{equ:geomSum},
we obtain
\begin{multline*}
\iota_\bfs  H^\star_n[\{2\}^b,c,\bfs]=
\sum_{(\bfp;\wbfp; \wwbfp)\in\Pi_1}\calH_n[(\ol{2b+2})\oplus p_1, \bfp_\hone; b+1+\wdp_1, \wbfp_\hone; 1\boxplus \wwdp_1,\wwbfp_\hone] \\
+\frac{q^{n(b+1)}}{[n]^{2b+c}}\sum_{(\bfp;\wbfp; \wwbfp)\in\Pi(\bfs)}\calH_n[\bfp; \wbfp; \wwbfp] -
\frac{q^{n(b+1)}}{[n]^{2b+2}}\sum_{(\bfp;\wbfp; \wwbfp)\in\Pi_1}\calH_n[\ol{p}_1, \bfp_\hone; \wbfp; 1\boxplus \wwbfp_\hone],
\end{multline*}
which implies
\begin{equation} \label{equ:Subg}
\begin{split}
\iota_\bfs H^\star_n[\{2\}^b,c,\bfs]=\sum_{(\bfp;\wbfp; \wwbfp)\in\Pi_{\bt^bc}}\calH_n[\bfp; \wbfp; \wwbfp]
&-\frac{q^{n(b+1)}}{[n]^{2b+2}}\sum_{(\bfp;\wbfp; \wwbfp)\in\Pi_2}\calH_n[\bfp; \wbfp; \wwbfp] \\
&+\frac{q^{n(b+1)}}{[n]^{2b+c}}\sum_{(\bfp;\wbfp; \wwbfp)\in\Pi(\bfs)}\calH_n[\bfp; \wbfp; \wwbfp],
\end{split}
\end{equation}
where $\Pi_2=\{\ol{0}\circ 1^{\circ(c-3)}\circ(\gl_1\oplus \ol{1})\circ\bfgl_\hone^\circ; 0^{\circ(c-2)}\circ\wbfgl{}^\circ;
1\circ\theta^{\circ(c-3)}\circ(\wwdl_1\boxplus (-1))\circ(\wwbfgl{}_\hone)^\circ\}.$
Expanding the second sum from \eqref{equ:Subg}, we have
\begin{equation*}
  \sum_{(\bfp;\wbfp; \wwbfp)\in\Pi_2}\calH_n[\bfp; \wbfp; \wwbfp] =
\sum_{(\bfp;\wbfp; \wwbfp)\in\Pi_3}
\sum_{\substack{
\bfw=\ol{0}\circ 1^{\circ(c-3)}\circ(p_1\oplus \ol{1})\\
\wbfw=0^{\circ(c-2)}\circ \wdp_1; \,  \wwbfw=1\circ \theta^{\circ(c-3)}\circ \wwdp_1
}}
\calH_n[\bfw, \bfp_\hone; \widetilde{\bfw}, \wbfp_\hone; \wwbfw, \wwbfp_\hone],
\end{equation*}
where $\Pi_3=\{\bfgl^\circ; \wbfgl{}^\circ;
(\wwdl_1\boxplus(-1)) \circ(\wwbfgl{}_\hone)^\circ\}.$
Applying Lemma~\ref{lem:combinatorial} to the inner sum with $a=p_1,$
$b=\wdp_1$, $r=\wwdp_1$, $c$ replaced by $c-2$, and $\bfx=\bfp_\hone$,
$\bfy=\wbfp_\hone$, $\bfz=\wwbfp_\hone$, we obtain
\begin{equation} \label{23}
\begin{split}
\sum_{(\bfp;\wbfp; \wwbfp)\in\Pi_2}\calH_n[\bfp; \wbfp; \wwbfp]
&=\sum_{(\bfp;\wbfp;\wwbfp)\in\Pi_3}\frac{1}{[n]^{c-2}}\,
\calH_n[p_1, \bfp_\hone; \wbfp; \wwdp_1\boxplus 1, \wwbfp_\hone]\\
&=\frac{1}{[n]^{c-2}}\sum_{(\bfp;\wbfp; \wwbfp)\in\Pi(\bfs)}\calH_n[\bfp; \wbfp; \wwbfp].
\end{split}
\end{equation}
To justify the last equality above, we need to show that for the components of $\wwbfgl$ satisfying~\eqref{equ:conditionWwgl} we have
$\wwdl_1\boxplus(-1)\boxplus 1=\wwdl_1$ and for any $j\ge 2,$
\begin{equation*}
\wwdl_1\boxplus(-1)\boxplus\wwdl_2\boxplus\dots\boxplus\wwdl_j\boxplus 1=\wwdl_1\boxplus\wwdl_2\boxplus\dots\boxplus\wwdl_j.
\end{equation*}
These can be proved by using the projection $\pi$ of Lemma~\ref{lem:boxplus} and the fact that
$\pi(\wwdl_1\boxplus(-1)\boxplus\dots\boxplus\wwdl_j\boxplus 1)=\pi(\wwdl_1\boxplus\dots\boxplus\wwdl_j)\in\{1,2\}$ by \eqref{equ:conditionWwgl}.

Now by  \eqref{equ:Subg} and  \eqref{23}, we see that  \eqref{equ:Rule2bc}
is true when $n+b=N+1$.  We have completed
the proof of the theorem.
\end{proof}

Repeatedly applying the theorem, we quickly find
\begin{cor} \label{cor:general}
Keep the same notation as in Theorem~\ref{thm:general}.
Then for any integers $a,b\ge 0$, $l\ge 1$ and $c\ge 3$, we have
\begin{equation}\label{equ:Rule2a1l}
H_n^{\star}[\{2\}^a,\{1\}^l, \bfs]=\iota_\bfs  \calH_{n}^\sharp[
2a+1,\{1\}^{l-1}, \bfgl;
a+1,\{1\}^{l-1},\wbfgl;
2,\{0\}^{l-1},\wwdl_1\boxplus -2,\wwbfgl{}_{\widehat{1}}]
\end{equation}
and
\begin{multline*}
H_n^{\star}[\{2\}^b, c, \{2\}^a,\{1\}^l,\bfs]= \iota_\bfs  \calH_{n}^\sharp[
 \ol{2b+2},\{1\}^{c-3},\ol{2a+2},\{1\}^{l-1},\bfgl; \\
 b+1,\{0\}^{c-3},a+1,\{1\}^{l-1},\wbfgl;
  1,\{\theta\}^{c-3},1,\{0\}^{l-1},\wwdl_1\boxplus-2,\wwbfgl{}_{\widehat{1}}].
\end{multline*}
\end{cor}
\begin{proof}
Repeatedly applying \eqref{equ:Rule2a1} by attaching $(2^{a_j},1)$, $j=1,\dots,l$
and then setting
$a_1=\dots=a_{l-1}=0$ and $a_l=a$,  we can quickly verify the \eqref{equ:Rule2a1l}.
The corollary follows by applying \eqref{equ:Rule2bc} to \eqref{equ:Rule2a1l}.
\end{proof}

We may take limit $n\to\infty$ in \eqref{equ:Rule2a} of Theorem~\ref{thm:general}
and Corollary~\ref{cor:general} to obtain identities for $q$-MZSV.
\begin{thm} \label{thm:general-qMZSV}
Let $\bfs=(s_1,\dots,s_d)\in\N^d$.
Set $\iota_\bfs =1$ if $s_d=1$ and $\iota_\bfs =-1$ if $s_d>1$.
Suppose $\bfs$ uniquely determines $[\bfgl;\wbfgl;\wwbfgl]$
satisfying \eqref{equ:conditionWwgl}
such that $\zeta^{\star}[\bfs]=\iota_\bfs  \frakz^\sharp[\bfgl;\wbfgl;\wwbfgl].$
Then for any integers $a,b\ge 0$, $l\ge 1$ and $c\ge 3$, we have
\begin{align*}
\zeta^{\star}[\{2\}^a,\bfs]  =\iota_\bfs \frakz^\sharp[& 2a\oplus\lambda_1,\bfgl{}_{\widehat{1}};
a+\wdl_1,\wbfgl{}{}_{\widehat{1}}; \wwbfgl],\\
\zeta^{\star}[\{2\}^a,\{1\}^l,\bfs] =\iota_\bfs \frakz^\sharp[&
2a+1,\{1\}^{l-1}, \bfgl;
a+1,\{1\}^{l-1},\wbfgl;
2,\{0\}^{l-1},\wwdl_1\boxplus -2,\wwbfgl{}_{\widehat{1}}],\\
\zeta^{\star}[\{2\}^b, c,\bfs]=\iota_\bfs  \frakz^\sharp[&
\ol{2b+2},\{1\}^{c-3},\gl_1\oplus \bar{1},\bfgl_{\widehat{1}};
b+1,\{0\}^{c-3},\wbfgl{};
1,\{\theta\}^{c-3},\wwdl_1\boxplus -1,\wwbfgl{}_{\widehat{1}}],\\
\zeta^{\star}[c, \{1\}^l,\bfs]= \iota_\bfs \frakz^\sharp[&
 \ol{2},\{1\}^{c-3},\ol{2},\{1\}^{l-1},\bfgl;  1,\{0\}^{c-3},\{1\}^{l},\wbfgl;
 1, \{\theta\}^{c-3},1,\{0\}^{l-1},\wwdl_1\boxplus-2,\wwbfgl{}_{\widehat{1}}].
\end{align*}
\end{thm}
\begin{proof}
The first three equations are straight-forward. The last one can be
obtained by applying the middle two equations successively after setting $a=b=0$.
\end{proof}

By letting $q\to 1$ in Theorem \ref{thm:general-qMZSV}
we can immediately prove Theorem~\ref{thm:general-MZSV} which gives the
corresponding general rule for  classical MZSV. Of course,
to guarantee convergence we need to restrict $a\ge 1$ there.

{}From Theorem~\ref{thm:general-qMZSV}, we can obtain a general
formula for arbitrary $q$-MZSV.

\begin{thm} \label{thm:general-qGlanois}
Let  $a_0,a_j\in \N_0$, $c_j\in \N$ and $c_j\ne 2$ for all $j=1,\dots,d$.
Set $\gd(c)=1$ if $c=1$ and $\gd(c)=0$ if $c\ge 3$. Moreover, put $\{\alpha\}^{n}=\{\alpha\}^{\max(n,0)}$.
Then we have
\begin{equation*}
\begin{split}
\zeta^{\star}[\{2\}^{a_0},c_1,\{2\}^{a_1},\dots,c_d,\{2\}^{a_d}]
= \pm \frakz^\sharp[&B_0,\{1\}^{c_1-3},B_1,\dots,\{1\}^{c_d-3},B_d; \\
&\widetilde{B}_0,\{0\}^{c_1-3},\widetilde{B}_1,\dots,\{0\}^{c_d-3},\widetilde{B}_d; \\
&\wwB_0,\{\theta\}^{c_1-3},\wwB_1,\dots,\{\theta\}^{c_d-3},\wwB_d].
\end{split}
\end{equation*}
Here the leading sign $\pm$ is $+$ if and only if $a_d=0$ and $c_d=1$,
\begin{align*}
B_j=
\left\{
  \begin{array}{ll}
    A_j, & \hbox{if $A_j$ is odd;} \\
    \ol{A_j}, & \hbox{if $A_j$ is even,}
     \end{array}
\right.
\quad \text{where} \quad
A_j=
\left\{
  \begin{array}{ll}
    2a_0+2-\gd(c_1) , & \hbox{if $j=0$;} \\
    2a_d+1-\gd(c_d), & \hbox{if $j=d$;} \\
    2a_j+3-\gd(c_j)-\gd(c_{j+1}), & \hbox{if $0<j<d$,}
  \end{array}
\right.
\end{align*}
\begin{align*}
\widetilde{B}_j=
\left\{
  \begin{array}{ll}
    a_j+1, & \hbox{if $0\le j<d$;} \\
    a_d, & \hbox{if $j=d$,}
     \end{array}
\right.
\, \text{and} \quad
\wwB_j=
\left\{
  \begin{array}{ll}
    1+\gd(c_1) , & \hbox{if $j=0$;} \\
    (1-\gd(c_d))\boxplus(-1), & \hbox{if $j=d$;} \\
   (1-\gd(c_j))\boxplus(\gd(c_{j+1})-1) , & \hbox{if $0<j<d$.}
  \end{array}
\right.
\end{align*}
Moreover, if $a_d=0$ and $c_d=1$, then $B_d$, $\widetilde{B}_d$, $\wwB_d$ are vacuous.
\end{thm}
\begin{proof}
The theorem can be proved easily by  induction on $d$ using
Theorem \ref{thm:general-qMZSV}. We leave the details to the interested reader.
\end{proof}
By letting $q\to 1$ in Theorem~\ref{thm:general-qGlanois}, we get Theorem~\ref{thm:general-Glanois} which gives the corresponding result
for classical MZSV. It is clear that to ensure  convergence we  need to assume that $a_0>0$ or $c_1\ge 3$.

\section{Some applications}
The first application gives us the general  $\oll{\bt\hyf c}\hyf\bt$ ($c\ge 3$) formula.
Here the underline means
the $(\{2\}^a,c)$-type string may be repeated an arbitrary number of times
where $a$ and $c$ may change in each repetition.
\begin{thm} \label{thm:MHS2c2}
Suppose $\ell\in \N_0$.
Let $\bfs=(\{2\}^{a_1},c_1,\dots,\{2\}^{a_\ell},c_\ell,\{2\}^{a_{\ell+1}})$
with $a_j, c_j\in\N_0$ and $c_j\ge 3$ for all $j\ge 1.$
Then
\begin{equation}\label{equ:GenMain2c2}
\begin{split}
 H^\star_n[\bfs]=- \calH_n^\sharp[
 &\ol{2a_1+2},\{1\}^{c_1-3},2a_2+3,\{1\}^{c_2-3},\dots,2a_\ell+3,\{1\}^{c_\ell-3},2a_{\ell+1}+1;\\
 & a_1+1,\{0\}^{c_1-3},\dots,a_\ell+1,\{0\}^{c_\ell-3},a_{\ell+1};\,
  1,\{\theta\}^{c_1+\cdots+c_\ell-2\ell}].
\end{split}
\end{equation}
\end{thm}
\begin{proof} If $a_{\ell+1}=0$, then starting from   Theorem~\ref{thm:MHS2c} for $H^\star_n[\{2\}^{a_\ell},c_\ell]$ and repeatedly applying
 \eqref{equ:Rule2bc}, we get the above identity. Otherwise,
starting from \eqref{equ:H2a} and repeatedly applying the attaching rule \eqref{equ:Rule2bc}
we can arrive at \eqref{equ:GenMain2c2} immediately.
\end{proof}

By applying Lemma~\ref{lem:last} to Theorem~\ref{thm:MHS2c2} we immediately get
\begin{cor} \label{thm:MZSV2c2}
With the same notation as in Theorem~\ref{thm:MHS2c2}, we have
\begin{equation*}
\begin{split}
 \zeta^\star[\bfs]=- \zeta^\sharp[
 &\ol{2a_1+2},\{1\}^{c_1-3},2a_2+3,\{1\}^{c_2-3},\dots,2a_\ell+3,\{1\}^{c_\ell-3},2a_{\ell+1}+1;\\
 & a_1+1,\{0\}^{c_1-3},\dots,a_\ell+1,\{0\}^{c_\ell-3},a_{\ell+1};\,
  1,\{\theta\}^{c_1+\cdots+c_\ell-2\ell}].
\end{split}
\end{equation*}
In particular, if $c_1=c_2=\ldots=c_\ell=3$, we get a $q$-analog of the Two-three formula:
\begin{equation*}
\zeta^\star[\bfs]
= \hskip-.3cm
 \sum_{\substack{
       \bfp=({2a_1+2})\circ(2a_2+3)\circ\cdots\circ(2a_\ell+3)\circ(2a_{\ell+1}+1) \\
       \widetilde{\bfp}=(a_1+1)\circ\cdots\circ(a_\ell+1)\circ(a_{\ell+1})
       } }
\ \sum_{k_1> \cdots >k_m\ge 1}\!\!(-1)^{k_1-1} q^{\frac{k_1(k_1-1)}{2}} \prod_{j=1}^m \frac{q^{\wdp_jk_j}(1+q^{k_j})}{[k_j]^{p_j}}.
\end{equation*}
\end{cor}
\begin{rem}
When $q\to 1$ one can recover all the MZSV identities
contained in \cite{LinebargerZh2013}.
\end{rem}

Now starting from \eqref{equ:H2a1} and repeatedly and alternatively applying the
attaching rules \eqref{equ:Rule2a1} and \eqref{equ:Rule2bc} we can find the following:
\begin{thm}\label{thm:MHS2c21&212c21}
Suppose $\ell\in\N_0$, $n\in \N$ and  $a_0, a_j, b_j, c_j-3\in\N_0$ for all $j\ge 1$.
Consider the following two possible types of compositions:
{\allowdisplaybreaks
\begin{align*}
& \!\!\!\!\!\!\!\!\!\!\!\!\!(\oll{\bt\hyf c\hyf\bt\hyf1}): \\
  \bfs&=(\{2\}^{b_1},c_1,\{2\}^{a_1},1,\dots,\{2\}^{b_\ell},c_\ell,\{2\}^{a_\ell},1), \,\,\, \ell\in\N, \\
  \bfs'&=(\ol{2b_1+2},\{1\}^{c_1-3},\ol{2a_1+2},\dots,
                            \ol{2b_\ell+2},\{1\}^{c_\ell-3},\ol{2a_\ell+2}; \\
  &  \qquad b_1+1,\{0\}^{c_1-3},a_1+1,\dots,b_\ell+1,\{0\}^{c_\ell-3},a_\ell+1; \\
  &  \qquad 1,\{\theta\}^{c_1-3},1, \underbrace{-1,\{\theta\}^{c_2-3}, 1
                            ,\dots,-1,\{\theta\}^{c_\ell-3},1}_{\text{appear only if $\ell>1$}}).\\
&\!\!\!\!\!\!\!\!\!\!\!\!\! (\bt\hyf1\hyf\oll{\bt\hyf c\hyf\bt\hyf1}): \\
 \bfs&=(\{2\}^{a_0},1,\{2\}^{b_1},c_1,\{2\}^{a_1},1,\ldots,\{2\}^{b_\ell},c_\ell,\{2\}^{a_{\ell}},1), \,\,\, \ell\in\N_0, \\
   \bfs'&=(2a_0+1, \ol{2b_1+2},\{1\}^{c_1-3},\ol{2a_1+2},\dots,
                         \ol{2b_\ell+2},\{1\}^{c_\ell-3},\ol{2a_\ell+2}; \\
 & \qquad    a_0+1,b_1+1,\{0\}^{c_1-3},a_1+1,\dots,b_\ell+1,\{0\}^{c_\ell-3},a_\ell+1; \\
  &   \qquad   2,\underbrace{-1,\{\theta\}^{c_1-3},1, \dots,-1,\{\theta\}^{c_\ell-3},1}_{\text{appear only if $\ell>0$}}).
\end{align*}
}
Then in each case we have
\begin{equation*}
 H_n^{\star}[\bfs]=\calH_n^\sharp[\bfs'].
\end{equation*}
\end{thm}

\begin{cor}\label{cor:MZSV2c21&212c21}
With the same notation as in Theorem~\ref{thm:MHS2c21&212c21}, we have
\begin{equation*}
 \zeta^{\star}[\bfs]= \frakz^\sharp[\bfs'].
\end{equation*}
\end{cor}

For example, taking $\ell=1$ and $c_1=3$, we get
(cf.\ \cite[(26)]{Zhao2013a} and the identity after it)
\begin{equation*}
\zeta^{\star}[\{2\}^b,3,\{2\}^a,1]=\frakz[2a+2b+4; a+b+2; 2]+\frakz[\ol{2b+2}, \ol{2a+2}; b+1, a+1; 1, 1]
\end{equation*}
and
\begin{equation}\label{equ:212321}
\begin{split}
\zeta^{\star}[\{2\}^{a_0}, 1, &\{2\}^b,3,\{2\}^{a_1},1]=\frakz[2(a_0+b+a_1)+5; a_0+b+a_1+3; 2] \\
&+\frakz[2a_0+1, 2a_1+2b+4; a_0+1, a_1+b+2; 2,0] \\
&+\frakz[\ol{2a_0+2b+3}, \ol{2a_1+2}; a_0+b+2, a_1+1; 1, 1] \\
&+\frakz[2a_0+1, \ol{2b+2}, \ol{2a_1+2}; a_0+1, b+1, a_1+1; 2, -1, 1].
\end{split}
\end{equation}
We can also get the following identity which is the
$q$-analog of \cite[Theorem~6.1(i)]{Zhao2013a}.
\begin{cor}
Let $a, b$ be two nonnegative integers. Then
\begin{equation*}
\begin{split}
\zeta^{\star}[\{2\}^a,3,\{2\}^b,1]+\zeta^{\star}[\{2\}^b,3,\{2\}^a,1]&=\zeta^{\star}[\{2\}^{a+1}]\zeta^{\star}[\{2\}^{b+1}] \\
&\quad +(1-q)\frakz[2a+2b+3; a+b+2; 2].
\end{split}
\end{equation*}
\end{cor}
\begin{proof} By taking $n\to\infty$ in \eqref{equ:H2a} and using Lemma~\ref{lem:last} we get
\begin{equation*}
\zeta^{\star}[\{2\}^{a+1}]=\frakz[2a+2; a+1; 1]=\sum_{k=1}^\infty \frac{q^{(a+1)k+Q(1,k)}(1+q^k)}{[k]^{2a+2}}.
\end{equation*}
Thus
\begin{align*}
&\zeta^{\star}[\{2\}^a,3,\{2\}^b,1]+\zeta^{\star}[\{2\}^b,3,\{2\}^a,1]
-\zeta^{\star}[\{2\}^{a+1}]\zeta^{\star}[\{2\}^{b+1}]\\
=&\sum_{k=1}^\infty \frac{q^{(a+b+2)k+Q(2,k)}\big(2(1+q^k)-(1+q^k)^2\big)}{[k]^{2a+2b+4}}\\
=&\sum_{k=1}^\infty \frac{q^{(a+b+2)k+Q(2,k)}(1+q^k)(1-q^k)}{[k]^{2a+2b+4}}\\
=&(1-q)\frakz[2a+2b+3; a+b+2; 2]
\end{align*}
as desired.
\end{proof}

If we start with \eqref{equ:H2a} and repeatedly and alternatively apply the
attaching rules \eqref{equ:Rule2a1} and \eqref{equ:Rule2bc} we can get:
\begin{thm}\label{thm:MHS2c212&212c212}
Suppose $\ell\in\N_0$, $n, a_{\ell+1}\in \N$, and  $a_0, a_j, b_j, c_j-3\in\N_0$ for all $1\le j\le \ell$.
Consider the following two possible types of compositions:
%\item[$:$]
{\allowdisplaybreaks
\begin{align*}
& \!\!\!\!\!\!\!\!\!\!\!\!\!(\oll{\bt\hyf c\hyf\bt\hyf1}\hyf\bt): \\
  \bfs&=(\{2\}^{b_1},c_1,\{2\}^{a_1},1,\dots,\{2\}^{b_\ell},c_\ell,\{2\}^{a_\ell},1,\{2\}^{a_{\ell+1}}), \,\,\, \ell\in\N, \\
  \bfs'&=(\ol{2b_1+2},\{1\}^{c_1-3},\ol{2a_1+2},\dots,
              \ol{2b_\ell+2}, \{1\}^{c_\ell-3},\ol{2a_\ell+2},\ol{2a_{\ell+1}}; \\
  &   \qquad b_1+1,\{0\}^{c_1-3},a_1+1,\dots,b_\ell+1,\{0\}^{c_\ell-3},a_\ell+1,a_{\ell+1}; \\
  &  \qquad  1,\{\theta\}^{c_1-3},1, \underbrace{-1,\{\theta\}^{c_2-3},1,\dots,-1,\{\theta\}^{c_\ell-3}, 1}_{\text{appear only if $\ell>1$}},-1). \\
&\!\!\!\!\!\!\!\!\!\!\!\!\! (\bt\hyf1\hyf\oll{\bt\hyf c\hyf\bt\hyf1}\hyf\bt): \\
 \bfs&=(\{2\}^{a_0},1,\{2\}^{b_1},c_1,\{2\}^{a_1},1,\ldots,\{2\}^{b_\ell},c_\ell,
        \{2\}^{a_{\ell}},1,\{2\}^{a_{\ell+1}}), \,\,\, \ell\in\N_0, \\
   \bfs'&=(2a_0+1,\ol{2b_1+2},\{1\}^{c_1-3},\ol{2a_1+2},\cdots ,                            \ol{2b_\ell+2},\{1\}^{c_\ell-3},\ol{2a_\ell+2},\ol{2a_{\ell+1}}; \\
 & \qquad  a_0+1,b_1+1,\{0\}^{c_1-3},a_1+1,\dots,b_\ell+1,\{0\}^{c_\ell-3},a_\ell+1,a_{\ell+1}; \\
 & \qquad 2,\underbrace{-1,\{\theta\}^{c_1-3},1,\dots,-1,\{\theta\}^{c_\ell-3}, 1}_{\text{appear only if $\ell>0$}},-1).
\end{align*}
}
Then in each case we have
\begin{equation*}
 H_n^{\star}[\bfs]=-\calH_n^\sharp[\bfs'].
\end{equation*}
\end{thm}
By taking $n\to \infty$ we have
\begin{cor}\label{thm:MZSV2c212&212c212}
Let notation be the same as in Theorem~\ref{thm:MHS2c212&212c212}. Then
\begin{equation*}
 \zeta^{\star}[\bfs]=-\frakz^\sharp[\bfs'].
\end{equation*}
\end{cor}
For example, taking $\ell=1$ and $c_1=3$, we get in case $(\oll{\bt\hyf c\hyf\bt\hyf1}\hyf\bt)$
\begin{equation*}
\begin{split}
\zeta^{\star}[\{2\}^{b},3,\{2\}^{a_1},1,\{2\}^{a_2}]=&-\frakz[\ol{2a_1+2b+2a_2+4}; a_1+b+a_2+2; 1] \\
&-\frakz[\ol{2b+2}, 2a_1+2a_2+2; b+1, a_1+a_2+1; 1,\theta] \\
&-\frakz[2a_1+2b+4, \overline{2a_2}; a_1+b+2, a_2; 2, -1] \\
&-\frakz[\ol{2b+2}, \overline{2a_1+2}, \overline{2a_2}; b+1, a_1+1, a_2; 1, 1, -1].
\end{split}
\end{equation*}
By taking $q\to 1$ this yields the identity on the bottom of \cite[p.~12]{Zhao2013a}.

As a non-trivial example of Theorem \ref{thm:general-qMZSV} we may attach a string of
type $(2^a,1)$ to the front of the already treated type $(\{2\}^{b},1,\{2\}^{c},3,\{2\}^{d},1)$
given by \eqref{equ:212321} and get the following $q$-MZSV identity:
for any nonnegative integers $a,b,c,d$
\begin{equation}\label{equ:q-keyExample}
\begin{split}
\zeta^{\star}[\{2\}^{a},1,&\{2\}^{b},1,\{2\}^{c},3,\{2\}^{d},1]=
      \frakz[2a+2b+2c+2d+6;a+b+c+d+4;2]\\
+&\frakz[2a+1,2b+2c+2d+5;a+1,b+c+d+3;2,0]\\
      +&\frakz[2a+2b+2,2c+2d+4;a+b+2,c+d+2;2,0]\\
+&\frakz[\ol{2a+2b+2c+4},\ol{2d+2};
                     a+b+c+3,d+1;
                     1,1]\\
      +&\frakz[2a+1,2b+1,2c+2d+4;a+1,b+1,c+d+2;2,0,0]\\
+&\frakz[2a+1,\ol{2b+2c+3},\ol{2d+2};
                     a+1,b+c+2,d+1;
                     2,-1,1]\\
+&\frakz[2a+2b+2,\ol{2c+2},\ol{2d+2};
                     a+b+2,c+1,d+1;
                     2,-1,1]\\
+&\frakz[2a+1,2b+1,\ol{2c+2},\ol{2d+2};
                     a+1,b+1,c+1,d+1;
                     2,0,-1,1].
\end{split}
\end{equation}

By taking $q\to 1$ in \eqref{equ:q-keyExample}
we discover the classical MZSV identity \eqref{equ:keyExample}
in the introduction, which has not been
proved before.

\end{document}